\documentclass[a4paper,11pt]{article}
\usepackage{comment}
\usepackage{amsmath}
\usepackage{amssymb}
\usepackage{amsthm}
\usepackage{graphicx}
\usepackage{enumerate}

\usepackage{ tipa }
\usepackage{tikz}
\usepackage{tikz-cd}
\usepackage{hyperref}
\usepackage{cleveref}
\usepackage[
backend=biber,
style=alphabetic,
sorting=nyt
]{biblatex}
\bibliography{bibliography.bib}

\newtheorem{theorem}{Theorem}[section]
\newtheorem{proposition}[theorem]{Proposition}
\newtheorem{lemma}[theorem]{Lemma}
\newtheorem*{lemma*}{Lemma}
\newtheorem{corollary}[theorem]{Corollary}

\newtheorem{definition}[theorem]{Definition}

\newtheorem*{claim}{Claim}

\newtheorem*{note}{Note}
\newtheorem{fact}[theorem]{Fact}

\AddToHook{env/lemma/begin}{\crefalias{theorem}{lemma}}
\AddToHook{env/observation/begin}{\crefalias{theorem}{observation}}
\AddToHook{env/proposition/begin}{\crefalias{theorem}{proposition}}
\AddToHook{env/definition/begin}{\crefalias{theorem}{definition}}
\AddToHook{env/corollary/begin}{\crefalias{theorem}{corollary}}
\AddToHook{env/fact/begin}{\crefalias{theorem}{fact}}

\crefname{corollary}{corollary}{corollaries}

\newenvironment{claimproof}{\proof} {\endproof} 

\renewcommand{\iff}{\Longleftrightarrow}
\renewcommand{\implies}{\Longrightarrow}
\newcommand{\cs}{\textopencorner}
\newcommand{\ce}{\textcorner}
\newcommand{\dcl}{\textrm{dcl}}
\newcommand{\acl}{\textrm{acl}}
\newcommand{\nl}{\newline}
\newcommand{\tx}{\textrm}

\newcommand{\Sn}{\mathcal{S}_n}
\newcommand{\Tn}{\mathcal{T}_n}
\newcommand{\Spn}{Sp_n}

\renewcommand{\dim}{\tx{dim}}
\makeatletter
\tagsleft@true\let\veqno\@@leqno

\tikzcdset{scale cd/.style={every label/.append style={scale=#1},
		cells={nodes={scale=#1}}}}

\begin{document}
	
	\title{Definable Functions to Quotients in Ordered Abelian Groups}
	\author{Harper Wells}
	\date{}
	\maketitle
	\begin{abstract}
    In this paper we study definable families of functions from an ordered abelian group into various naturally arising definable quotients. We show that for an ordered abelian group $G$ and definable family of convex subgroups $\{D\}_{D\in\mathcal{D}}$, any definable family of functions $\{f_D\} _{D\in\mathcal{D}}$ with $f_D:G^d\rightarrow\frac{G}{D}$ is uniformly piecewise linear; for a prime $p$, integers $s,r\geq 1$, and groups $D^{[p^s]}$ defined later, if $f_D:G^d\rightarrow\frac{G}{D+p^rG}$ or $f_D:G^d\rightarrow\frac{G}{D^{[p^s]}+p^rG}$ we instead obtain that the definable family of functions is uniformly piecewise a boolean combination of linear functions to quotients by subgroups which are uniformly definable from $D$.
    \end{abstract}

\section{Introduction}
The study of the model theory of ordered abelian groups goes back to at least
the sixties, starting with Robinson and Zakon in \cite{RobinsonZakon}
classifying regular ordered abelian groups, and Gurevich expanding on this and
studying decidability of ordered abelian groups in \cite{Gurevic}. Quantifier
elimination in certain cases was studied by Weispfenning in
\cite{WeispfenningQE}, and a relative quantifier elimination for all ordered
abelian groups was introduced by Schmitt in \cite{SchmittHab}. Cluckers and
Halupczok more recently proved a relative quantifier elimination in \cite{QE}
which uses the imaginary expansion of the group, rather than Schmitt's result
involving the spines as an external structure. 
\nl\nl
The study of definable functions in ordered abelian groups has been studied in the case of groups of finite regular rank by Belegradek, Verbovskiy, and Wagner in \cite{CosetMinimal}, and then further by Cluckers and Halupczok in their quantifier elimination paper. For Cluckers and Halupczok, the motivation arose from studying definable functions in Henselian valued fields. The main goal of this paper is to study definable functions not just from the group to itself, but from the group to various definable quotients; the motivation for this arose from questions regarding elimination of imaginaries in a general ordered abelian group.
\nl\nl
When it comes to elimination of imaginaries, work of Vicar{\'{\i}}a in \cite{Mariana} deals with the case of bounded regular rank; in \cite{ImaginariesProductsAndAdeleRing}, Derakhsan and Hrushovski prove a result for elimination of imagninaries in general products of structures, while in \cite{noEI} Liccardo proves a negative result regarding elimination of imaginaries in pure ordered abelian groups. While extensive, these results still don't constitute a complete image of elimination of imaginaries in ordered abelian groups; namely, not all ordered abelian groups are elementarily equivalent to products of ordered abelian groups of bounded regular rank, and even in the case of such products there is potential to improve the collection of sorts used to obtain elimination of imaginaries.
\nl\nl
One method for proving weak elimination of imaginaries down to a specified collection of sorts involves proving density of definable types, and finding codes for definable types in those sorts. A proof of this can be found in \cite{ProofOfEI}. This second condition can also be broken down further into coding definable $1$-types and coding germs of definable functions into (finite subsets of) the sorts to which imaginaries are to be eliminated; this was the motivation for this project.
In the aim of showing that germs of definable functions have codes in appropriate sorts, we show that uniform families of definable functions from the group into natural quotient sorts are given piecewise by uniform intersections of linear functions into certain other quotients which are uniformly definable from the target quotient. This then allows us to code the uniform family of functions with a combination of a definable set in one of the spines and a tuple of elements coming from quotient groups of the form we study in this paper.
\nl\nl
Specifically, we classify definable functions $f:G^d\rightarrow \frac{G}{H}$ with $H = D$, $H=D+p^rG$, and $H=\alpha^{[p^s]}+p^rG$ where $D$ is a convex subgroup, $p$ is prime, and $\alpha$ is a member of one of the spines; for the definition of spines and $\alpha^{[p^s]}$ see \Cref{SnTnTn^+,def D^{[n]}}. For $H=D$ or $H=D+p^rG$ we have the following:

\begin{theorem}\label{introTheorem}
	
	Suppose either $H = D$ or $H= D+p^rG$ for some $D\leq G$ convex and $r\geq 1$. Also suppose that $d,e\in\mathbb{N}_{>0}$, and $f:G^d\rightarrow (\frac{G}{H})^{(e)}$ is a $B$-definable function, where by $X^{(e)}$ we mean subsets of $X$ of size $e$. Then we can find some $k\in \mathbb{N}$, $\acl^{eq}(B)$-definable sets $\{X_i\subseteq G^d\}_{i\leq k}$, $\acl^{eq}(B)$-definable linear functions $f_i:X_i\rightarrow \frac{G}{H}$ such that $f = \bigcup\limits_{i\leq k}f_i$.
\end{theorem}
\noindent For $H=\alpha^{[p^s]}+p^rG$ the result is a bit more technical; for a full breakdown of the individual function case see \Cref{(theorem) main for a^p^s+p^rG}. These results are also immediately uniformisable in the case where $H=D$ and $H=\alpha^{[p^s]}+p^rG$, while some modifications are required for the uniform case for $H=D+p^rG$.
\nl\nl
In \Cref{prelims} we go over some required definitions and additional preliminaries, while in \Cref{main results} we state the non-uniform versions of the main results in full. In \Cref{convex section,alpha^{[p^s]}+p^rG section,D+p^rG section} we state and prove the uniform versions of the theorems, obtaining the non-uniform versions as corollaries. Finally, in \Cref{examples section} we give various examples showing why we cannot have an immediate uniform adaptation of \Cref{introTheorem} for all of the various cases.

\section{Preliminaries and Notation}\label{prelims}

\subsection*{Ordered Abelian Groups}
In this section, we cover various preliminaries regarding the model-theoretic treatment of ordered abelian groups (oags). First, we recall that when we consider $(\mathbb{Z}, +, 0, <)$ as an oag, we obtain quantifier elimination in the language of Presburger arithmetic, $L_{Pres} := \{0, 1, +, -, <, \{\equiv_p\}_{p\in\mathbb{P}} \}$, where $\mathbb{P}$ is the set of primes, and $a\equiv_p b \iff (\exists z)(a-b = pz)$.
\nl\nl
From this, we see that divisibility essentially controls the model theory of $(\mathbb{Z}, +, 0, <)$, and the same turns out to hold true for a general ordered abelian group, although we have to take extra care when considering what is meant by divisibility. In particular, we have the following definitions:
\begin{definition}\label{SnTnTn^+}
	Let $(G, 0, +, <)$ be an oag, and $g\in G,n\in\mathbb{N}_{>0}$. Then we have the following definitions:
	\begin{enumerate}[(1)]
		\item Let $H_{n,g}$ be the largest convex subgroup of $G$ such that $g\notin H_{n,g} + nG$. We define $\Sn := G/\sim$, where $a\sim b\iff H_{n,a} = H_{n,b}$, and let $s_n:G\rightarrow\Sn$ be the canonical projection map. If $g\in nG$ then we set $H_{n,g} = \{0\}$.
		\item Let $H'_{n,g} = \bigcup\limits_{\substack{\alpha\in\Sn\\g\notin\alpha}}\alpha$, where the empty union is $\{0\}$. We define $\Tn := G/\sim'$, where $a\sim' b\iff H'_{n,a} = H'_{n,b}$, and $t_n:G\rightarrow\Tn$ is the canonical projection map
		\item Let $H'^+_{n,g} = \bigcap\limits_{\substack{\alpha\in\Sn \\g \in\alpha}}\alpha$. We define $\Tn^+ := G/\sim^+$, where $a\sim^+b \iff H'^+_{n,a}=H'^+_{n,b}$, and $t_n^+:G\rightarrow \Tn^+$ is the canonical projection map
		\item The $n$-spine of $G$ is $\Spn = \Sn\cup\Tn\cup\Tn^+$, which forms a preorder via inclusion of convex subgroups; this collapses to a total order after equating elements which represent the same subgroup.
	\end{enumerate}
\end{definition}
\begin{note}
	Each of $H_{n,g}, H'_{n,g}, H'^{+}_{n,g}$ are uniformly definable convex subgroups of $G$, hence each of $\Sn, \Tn,\Tn^+$ is an imaginary sort of $G$. Throughout this paper (including in the above definitions), we conflate equivalence classes in the $n$-spines with the subgroup they correspond to for notational convenience.
\end{note}

\begin{definition}\label{def D^{[n]}}
	For $D\leq G$ be convex, $n\in\mathbb{N}_{>0}$, we define
    \begin{align*}
        D^{[n]} = \bigcap\limits_{\substack{D\leq H\leq G\\ H \tx{ convex}}}H+nG
    \end{align*}
    By \cite{QE}, we have that
    \begin{align*}
        D^{[n]} = \bigcap\limits_{\substack{D\leq \alpha\leq G\\ \alpha \in\Sn}}\alpha+nG
    \end{align*}
\end{definition}
\begin{note}
	Suppose $g\in D^{[n]}\setminus D$. Then $s_n(g) = D$, hence $D \subsetneq D^{[n]}\implies D\in \Sn$.
\end{note}
\begin{fact}\label{Functions to G are linear}
	For any function $f:G^n\rightarrow G$ which is $L_{oag}$-definable with parameters from $B$, there exists a partition of $G$ into finitely many $B$-definable sets such that on each set $A$, $f$ is linear: there exists $r_1,\dots, r_n, s,k\in \mathbb{Z}$ with $s\neq 0$, $b\in\dcl(B)$ with $\sum_ir_ix_i + b$ a term appearing in the definition for $f$, and $k$ bounded by $N\in\mathbb{N}$ depending only on the formula used to define $f$, such that for any $\bar a\in A$, we have $f(a_1,\dots,a_n) = \frac{1}{s}(\sum_i r_ia_i + b+k)$.
\end{fact}
\begin{proof}
	The proof for this is exactly the proof of Corollary~1.10 of \cite{QE} 
\end{proof}
\subsection*{General Notation}
Throughout this paper we will make use of various other potentially nonstandard notation. First, there is a notational clash between the usual notation for the set of subsets of a fixed finite size and the subgroup defined in \Cref{def D^{[n]}}, so given a set $X$, and $n\in \mathbb{N}_{\geq 1}$ we have
\begin{align*}
    X^{(n)} = \{A\subseteq X : |A| = n\}
\end{align*}
Also, throughout this paper we deal with functions from a group $G$ to various quotients. If we have functions $f: G^d\rightarrow \frac{G}{H}, g: G^d\rightarrow \frac{G}{K}$ then these determine a (partial) function $h: G^d\rightarrow \frac{G}{H\cap K}$, where $h(\bar x)=y+H\cap K\iff f(\bar x) = y + H \wedge g(\bar x) = y+K$. We will write $h = f\cap g$. Note that if we consider these functions as relations on $G^{d+1}$ then this is literally true. We also make use of this identification for other boolean operations.
\nl\nl
Finally, if we have a linear function $f:G^d\rightarrow G$ and some $H\leq G$ which is either of the form $D+nG$ or $D^{[m]}+nG$, which will be of the form $f(\bar x) = \frac{1}{s}(\sum r_ix_i + g + k)$ with $s,r_i,k\in\mathbb{Z}$, we will write $f^H(\bar x) = \frac{1}{r}(\sum r_ix_i + g + n_{D})+H$, where $1_D$ is the smallest positive coset of $\frac{G}{D}$ if it exists and $0$ otherwise, and $k_D = k\cdot1_D$. Note that this is \textbf{not} simply the projection to $\frac{G}{H}$.

\newpage
\section{Definable functions from an ordered abelian group into a quotient}\label{main results}
A key object of study in any structure is the class functions definable in said structure. In \cite{QE}, it is shown that for $G$ an oag, any definable function $f:G^n\rightarrow G$ is piecewise linear. Note, however, that the quantifier elimination result for oags makes use of the expanded structure of the group along with sorts for the $n$-spines of the group, and we also add symbols for Presburger arithmetic relative to the quotients by elements of the $n$-spines, noting that when adding symbols for Presburger arithmetic we may restrict to only adding symbols for primes and prime powers. As such, we would like to understand the definable functions $f:G^d\rightarrow \frac{G}{E_\alpha}$, where $E$ is an equivalence relation given by one of the relative Presburger symbols.
\nl\nl
Note that $E_\alpha$ is given by cosets of the following subgroups:
\begin{itemize}
	\item $\alpha$ for some $\alpha\in Sp_n$ for some $n$
	\item $\alpha + p^rG$ with $\alpha\in \Sn$ with $r\geq 1$ and $p$ prime
	\item $\alpha^{[p^s]}+p^rG$ for some $\alpha\in \Sn$ with $s\geq r\geq 1$ and $p$ prime
\end{itemize}
We also expand our focus by allowing $\alpha$ to be any definable convex subgroup; note however that if $g\in D^{[p^s]}\setminus (D+p^sG)$, then $s_{p^s}(g) = D$ and so $D\in \mathcal{S}_{p^s}$.
\nl\nl
We now state the main two theorems of the paper, which also treat functions of the form $f:G^d\rightarrow(\frac{G}{H})^{(e)}$. For the case where $H$ is either convex, or $H= D+p^rG$ with $D$ convex, we get that $f$ is piecewise linear:

\begin{theorem}\label{main}
	
	Suppose either $H = D$ or $H= D+p^rG$ for some $D\leq G$ convex and $r\geq 1$. Also suppose that $d,e\in\mathbb{N}_{>0}$, and $f:G^d\rightarrow (\frac{G}{H})^{(e)}$ is a $B$-definable function. Then we can find some $k\in \mathbb{N}$, $\acl^{eq}(B)$-definable sets $\{X_i\subseteq G^d\}_{i\leq k}$, $\acl^{eq}(B)$-definable linear functions $f_i:X_i\rightarrow \frac{G}{H}$ such that $f = \bigcup\limits_{i\leq k}f_i$.
\end{theorem}
\begin{note}
	We may assume that $[G:H]=\infty$, since otherwise the theorem is immediate by taking $f_i$ to be constant and $X_i = \{\bar x:f(\bar x) = f_i(\bar x)\}$
\end{note}
\noindent However, the case where $H= \alpha^{[p^s]}+p^rG$ it is not quite as simple, and requires the following technical modification:
\newpage
\begin{theorem}\label{(theorem) main for a^p^s+p^rG}
	Let $d,e\in\mathbb{N}_{>0}$, $\alpha\in \Sn$ for some $n$, $1\leq r \leq s$, $\alpha^{[p^s]}+p^rG \neq D+p^rG$ for any convex $D\leq G$, and $f:G^d\rightarrow (\frac{G}{\alpha^{[p^s]}+p^rG})^{(e)}$ be a $B$-definable function. Then we can find:
	\begin{itemize}
		\item $k, n_j, m_j\in \mathbb{N}$,  $B$-definable sets $\{X_j\subseteq G^d\}_{j\leq k}$
		\item $1\leq \ell_{j,1},\dots,\ell_{j,n_j}\leq s$, $1 \leq i_{j,1},\dots,i_{j,n_j}\leq r$ with 
        \begin{align*}
            \left[\bigcap\limits_{1\leq t\leq n_j}(\alpha^{[p^{\ell_{j,t}}]}+p^{i_{j,t}G}):\alpha^{[p^s]}+p^rG\right]<\infty
        \end{align*}
		\item $B$-definable linear functions $f_{j,t}:G^d\rightarrow G$, $g_{j,t}:G^d\rightarrow G$
		\item $1\leq v_{j,1},\dots,v_{j,m_j}\leq s$, $1\leq w_{j,1},\dots,w_{j,m_j}\leq r$
	\end{itemize} such that $f = \bigcup\limits_{j\leq k} ((\bigcap\limits_{t\leq n_j}f_{j,t}^{H_{j,t}}|_{X_j})\setminus(\bigcup\limits_{t\leq m_j}g_{j,t}^{K_{j,t}}|_{X_j}))$, where $H_{j,t}=\alpha^{[p^{\ell_{j,t}}]}+p^{i_{j,t}}G$ and $K_{j,t} = \alpha^{[p^{v_{j,t}}]}+p^{w_{j,t}}G$
\end{theorem}
\noindent These theorems concern functions of the form $f:G^d\rightarrow (\frac{G}{H})^{(e)}$ for various forms of definable subgroup $H\leq G$; we in fact prove some stronger uniform versions of these theorems for definable families of functions $f_H:G^d\rightarrow (\frac{G}{H})^{(e)}$. These stronger versions are straightforward generalisations in the case where $H$ is either convex or of the form $\alpha^{[p^s]}+p^rG$, but the generalisation in the case where $H=D+p^rG$ for some definable convex $D\leq G$ is more immediately similar to \Cref{(theorem) main for a^p^s+p^rG} than \Cref{main}.

\section{Functions to quotients by definable convex subgroups}\label{convex section}
We start with this case as it is the most straightforward of the various possibilities; in this case, we can actually directly prove a uniform version of \Cref{main}, which is as follows:
\begin{theorem}\label{Uniform main for convex}
    For $B\subseteq G$, let $\mathcal{D}$ be a $B$-definable family of convex subgroups of $G$, and $$\mathcal{F}= \left\{f_D:G^d\rightarrow \frac{G}{D}\right\}_{D\in\mathcal{D}}$$ be a uniformly $B$-definable family of functions. Then $\mathcal{F}$ is uniformly piecewise linear.
    More precisely, there exists a $B$-definable family of partitions of $G^d$ into $l$ pieces, $\{X_i^D\}_{i\leq l}$, and $B$-definable linear functions $f_{i}:G^d\rightarrow G$, such that for all $D\in\mathcal{D}$
    \begin{align*}
        f_D = (\bigcup\limits_{i\leq l} f_{i}^D|_{X_i^D}):G^d\rightarrow \frac{G}{D}
    \end{align*}
\end{theorem}
\begin{note}
    In this formulation, we parameterise the family of definable subgroups by the sequence of imaginaries for the definable subgroups as opposed to parameterising by tuples from the group. Because of this, we may assume that the parameters from the definable family do not appear in formulas except in a location where a group may appear. Note also that we may still apply these theorems in cases where we parameterise by tuples in the model by considering the parameters as additional input variables in the function.
\end{note}
\begin{proof}
    First, note that we may add $1_D, \equiv_D, \equiv_{D,n}, <_D, \pi_D$ to our language, as these are just expansions by definitions; these are symbols for Presburger arithmetic relative to $\frac{G}{D}$ and projection to $\frac{G}{D}$. We may also add a uniform family of predicates $\mathcal{D}_n = \{D_n\}_{D\in\mathcal{D}}$ to each $n$-spine, where $\alpha\in D_n\iff\alpha\leq D$
    \nl\nl
    Next, by \cite{QE}, we may assume (possibly at the cost of splitting into finitely many $B$-definable pieces) that $\mathcal{F}$ is defined by a formula of the form
    $$(\exists\bar\theta)(\xi(\bar\theta, \bar\eta,D)\wedge\psi(\bar x, \pi_D(y), \bar b, \bar\theta))$$
    where $\xi$ is a formula lying purely in $\Spn$ along with the aforementioned uniform family of predicates, and $\psi$ is a conjunction of literals.
    \nl\nl
    Now, by Section 4 of \cite{StableEmbeddedness}, $(D,\frac{G}{D})$ is stably embedded in $G$, and so by compactness (and allowing ourselves to split further into finitely many $B$-definable pieces) we can take this formula to be of the form
    $$(\exists\bar\theta)(\xi(\bar\theta, \bar\eta,D_n)\wedge\phi(\bar x, \bar b, \bar\theta)\wedge \chi(\pi_D(\bar x), \pi_D(y),\pi_D(\bar b), \pi_D(\bar\theta)))$$
    where $\chi$ is a formula in $\frac{G}{D}$ (in particular, here is where we need the inclusion of Presburger arithmetic relative to $\frac{G}{D}$ in the language), and $\pi_H(\bar\theta) = \bar\theta \cup H$.
    We may further take $\chi$ to be of the form given by quantifier elimination relative to the spines.
    \nl\nl
    But now, by the proof of \Cref{Functions to G are linear}, there are a finite number of $\mathbb{Z}$-linear terms $t_1(\bar w), s_1(\bar u), \dots, t_n(\bar w), s_n(\bar u)$, which may include $k_D$, and $r_1,\dots,r_n\in \mathbb{Z}\setminus\{0\}$ such that $\forall \bar w\in\frac{G}{D}, \Delta\in \Spn^{\frac{G}{D}}$ where $\chi(\bar u, v, \bar w, \bar\Delta)$ defines a function $g_D:\frac{G}{D}^{|\bar u|}\rightarrow\frac{G}{D}$, $g(\bar u) = \frac{1}{r_i}(s_i(\bar u) + t_i(\bar w))$ for some $i$. Specifically, these terms come from the inequalities present as literals within $\chi$.
    \nl\nl
    Hence, we see that $\chi(\pi_D(\bar x), \pi_D(y),\pi_D(\bar b), \pi_D(\bar\theta))$ is piecewise given by linear functions of the form
    $$r_i\pi_D(y) = s_i(\pi_D(\bar x)) + t_i(\pi_D(\bar b))$$
    which is equivalent to
    $$r_iy \equiv_D s_i(\bar x) + t_i(\bar b)$$
    But then we may take $f_i(\bar x) = \frac{1}{r_i}(s_i(\bar x) + t_i(\bar b))$ and then $f_D$ can be uniformly decomposed into a finite number of pieces, each of which is given by a uniform projection of $f_i$ as required.

\end{proof}

\section{Functions to quotients by subgroups of the form $\alpha^{[p^s]}+p^rG$, with $\alpha\in \Sn$ for some $n$}\label{alpha^{[p^s]}+p^rG section}

In this section, we prove a uniform version of \Cref{(theorem) main for a^p^s+p^rG}, which turns out to be a rather direct uniform adaptation. We do this case first since proving the uniform version for subgroups of the form $D+p^rG$ with $D$ convex uses an induction in which it is helpful to be able to reference this in the inductive step.
\nl\nl
We start with noting a useful reduction: we may assume $r\leq s$, since by \cite{QE} if $r<s$ then $\alpha^{[p^s]}+p^rG = \alpha^{[p^r]}+p^rG$. We also have some important preliminary lemmas to prove, starting with a lemma that is key in understanding the behaviour os subgroups of the form $\alpha^{[p^s]}$. We follow this by proving that if there is any convex subgroup $D$ such that $\alpha^{[p^s]}+p^rG = D+p^rG$, then we can find such a $D$ that is definable, and so definable functions to $\frac{G}{\alpha^{[p^s]}+p^rG}$ are covered by \Cref{D+p^rG section}; hence, we may safely assume that $\alpha^{[p^s]}+p^rG$ is not of this form.
\begin{lemma}\label{keylemma}
	Let $D\leq G$ be a convex subgroup, and $s\geq r$. Then $$D^{[p^s]}\cap p^rG = p^rD^{[p^{s-r}]}$$
\end{lemma}
\begin{proof}
	\begin{align*}
		 D^{[p^s]}\cap p^rG  &= p^rG\cap\bigcap\limits_{\substack{ \beta > D\\\beta\in \mathcal{S}_p}}\beta+p^{s}G \\
		&= \bigcap\limits_{\substack{ \beta > D\\\beta\in \mathcal{S}_p}}(p^rG\cap\beta)+p^{s}G & \tx{since } p^sG\subseteq p^rG\\
		&= \bigcap\limits_{\substack{ \beta > D\\\beta\in \mathcal{S}_p}}p^r\beta+p^{s}G & \tx{since }\beta\tx{ is a pure subgroup} \\
		&= \bigcap\limits_{\substack{ \beta > D\\\beta\in \mathcal{S}_p}}p^r(\beta+p^{s-r}G) \\
		&= p^r\bigcap\limits_{\substack{ \beta > D\\\beta\in \mathcal{S}_p}}\beta+p^{s-r}G & \tx{since }G\tx{ is torsion-free} \\
		&= p^rD^{[p^{s-r}]}
	\end{align*}
	
	
\end{proof}
\begin{lemma}\label{convex implies definable convex}
	Let $\alpha\in\Sn$. If $\alpha^{[p^s]}= D+p^sG$ for some $D\leq G$ convex, then there is some $\alpha$-definable convex $D\leq G$ such that $\alpha^{[p^s]} = D +p^sG$.
\end{lemma}
\begin{proof}
	By Lemma~2.4 of \cite{QE}, we have that either:
	\begin{align}
		&(\forall \beta\in \mathcal{S}_p)(\beta > \alpha \rightarrow (\beta+p^sG \neq\alpha^{[p^s]})) \\
		&(\exists\beta\in \mathcal{S}_p)(\beta > \alpha \wedge \beta + p^sG = \alpha^{[p^s]}
	\end{align}
	\setcounter{equation}{0}
	And in case (1), we may define $D$ by
	$$D = \bigcap\limits_{\substack{\beta > \alpha\\\beta\in \mathcal{S}_p }}\beta$$
	while in case (2) we may take $D$ to be defined as:
	$$D = \bigcup\limits_{\substack{\beta > \alpha\\\beta\in \mathcal{S}_p \\ \beta + p^sG = \alpha^{[p^s]}}}\beta$$
\end{proof}
\noindent
Now, we prove an easy to check condition which determines whether $\alpha^{[p^s]}$ is of the form $D+p^sG$ for some convex $D$, which will be used later to transfer this result between different values of $s$; this will then be useful for studying certain quotients involving $\alpha^{[p^s]}+p^rG$. It is worth noting that if $\alpha^{[p^s]} = D+p^sG$, then $\alpha^{[p^s]}+p^rG = D+p^rG$ since $r\leq s$, and since we are in the case where $\alpha^{[p^s]}+p^rG\neq D+p^rG$ for any convex subgroup $D$, the same holds for $\alpha^{[p^s]}$.

\begin{lemma}\label{nonconvex conditions}
	Suppose $\alpha\leq G$ is convex, then $\alpha^{[p^s]}\neq D+p^sG$ for any convex $D\leq G$ iff all of the following hold:
	\begin{enumerate}[(1)]
		\item $(\forall \gamma > \alpha)(\gamma\in \mathcal{S}_{p^s} \rightarrow (|\{\beta+p^sG :\beta\in\mathcal{S}_{p^s}\wedge (\gamma>\beta>\alpha)\}| =\infty))$
		\item $\alpha = \bigcap\limits_{\substack{\beta> \alpha \\ \beta \in \mathcal{S}_{p^s}}}\beta$
		\item $\alpha = s_{p^s}(a)$ for some $a\in G$
	\end{enumerate}
	
\end{lemma}
\begin{note}
	By Lemma~2.2 of \cite{QE}, we may replace $\mathcal{S}_{p^s}$ by $\mathcal{S}_p$, and $s_{p^s}(a)$ by $s_p(a)$
\end{note}
\begin{proof}
	First, let $H = \bigcap\limits_{\substack{\beta> \alpha \\ \beta \in \mathcal{S}_p}}\beta$. Note that $H+p^sG\subseteq \alpha^{[p^s]}$ by Lemma~2.4 of \cite{QE}. Also, if $H \neq \alpha$, $\alpha^{[p^s]}\subseteq H+p^sG $ by the definition of $\alpha^{[p^s]}$.
	\nl\nl
	Now, suppose that $\alpha^{[p^s]}\neq D+p^sG$ for any convex $D\leq G$. Then we must have $H = \alpha$, and so (2) holds. Also, (1) must hold, since if not then $\bigcap\limits_{\substack{\beta > \alpha\\\beta\in \mathcal{S}_p}}\beta+p^sG$ is equivalent to a finite intersection of descending groups, and is hence equal to one of these groups, but then $\alpha^{[p^s]} = \beta + p^sG$. Finally, let $x\in \alpha^{[p^s]}\setminus (\alpha+p^sG)$. Then for $\beta > \alpha$, $x\in \beta +p^sG$ by the definition of $\alpha^{[p^s]}$, so $s_{p^s}(x) = \alpha$.
	\nl\nl
	Conversely, (1) implies that for $\beta \in \mathcal{S}_p,\beta > \alpha$, $\alpha^{[p^s]}\subsetneq \beta+p^sG$, and (2) implies that if $D > \alpha$ is convex, then there is some $\beta\in \mathcal{S}_p, \alpha < \beta \leq D$, and so $\alpha^{[p^s]}\subsetneq\beta+p^sG\subseteq D+p^sG$. Also, if $s_{p^s}(a) = \alpha$, then $a\in \alpha^{[p^s]}$, but also $a\notin D+p^sG$ for any $D\leq G$ convex with $D\leq \alpha$. Hence, $\alpha^{[p^s]}\neq D+p^sG$ for any $D\leq G$ convex.
\end{proof}

\begin{lemma}\label{equal p-divisibility}
	Let $L\leq K\leq G$ be convex, $r\geq 1$. Then $$L + pG = K+pG\iff L+p^rG = K + p^rG$$
\end{lemma}
\begin{proof}
	First, suppose that $L+pG = K + pG$. We prove inductively that $K+p^rG= L+p^rG$. Note that
	$$\frac{G}{K+pG}\cong\frac{K+p^iG}{K+p^{i+1}G}$$ via the map $$x+K+pG\mapsto p^ix + K+p^{i+1}G$$
	which is injective as $K$ is convex, hence pure. It is also clearly surjective. Then, by induction we get that
	\begin{align*}
		\frac{G}{L+pG}&\cong \frac{L+p^iG}{K+p^{i+1}G}\\
		f:x+L+pG&\mapsto p^ix + K+p^{i+1}G
	\end{align*}
	But we also have that
	\begin{align*}
		\frac{G}{L+pG}&\cong \frac{L+p^iG}{L+p^{i+1}G}\\
		g:x+L+pG&\mapsto p^ix + L+p^{i+1}G
	\end{align*}
	As well as the surjection
	\begin{align*}
		\frac{L+p^iG}{L+p^{i+1}G}&\twoheadrightarrow\frac{L+p^iG}{K+p^{i+1}G} \\
		h:x + L+p^{i+1}G &\mapsto x + K + p^{i+1}G
	\end{align*}
	But then $h\circ g = f$, and so $h$ must be an isomorphism, and hence we conclude that $L+p^{i+1}G=K+p^{i+1}G$.
\end{proof}

\begin{corollary}\label{same number above}
	Suppose $\alpha< D\leq G$ are convex, then
	\begin{align*}
		|\{\beta+pG :\beta\in\mathcal{S}_p\wedge D>\beta>\alpha\}|= |\{\beta+p^sG :\beta\in\mathcal{S}_p\wedge D>\beta>\alpha\}|
	\end{align*}
\end{corollary}
\begin{proof}
	Immediate from \Cref{equal p-divisibility}
\end{proof}
\begin{corollary}
	Let $\alpha\in \mathcal{S}_p$. Then $\alpha^{[p^s]} \neq D+p^sG$ for any convex $D\leq G$ iff $\alpha^{[p]} \neq D+pG$ for any convex $D\leq G$
\end{corollary}
\begin{proof}
	By \Cref{same number above}, condition (1) of \Cref{nonconvex conditions} is independent of $s$. By the note after \Cref{nonconvex conditions}, so are (2) and (3).
\end{proof}


\begin{lemma}\label{infiniterightquotientsnonconvex}
	Let $\alpha <D\leq G$ be convex, $\alpha^{[p^s]}\neq K +p^sG$ for any convex $K\leq G$, $s\geq r\geq 1$. Then 
	\begin{align*}
		\infty=[\alpha^{[p^s]}+p^{r-1}D:\alpha^{[p^s]}+p^rD] &= [D:(\alpha^{[p^{s-r+1}]}\cap D)+pD] \\
		&\leq [D:(\alpha^{[p^s]}\cap D)+pD] = \infty
	\end{align*}
	In particular, $[\alpha^{[p^s]}+p^{r-1}G:\alpha^{[p^s]}+p^rG] =\infty$
\end{lemma}
\begin{proof}
	First, by \Cref{nonconvex conditions} and \Cref{same number above}, $|\{\beta+pG :\beta\in\mathcal{S}_p\wedge\alpha < \beta<D\}|=\infty$.
	\nl\nl
	Next we prove the following:
	\begin{claim}
		$\frac{\alpha^{[p^s]}+p^{r-1}D}{\alpha^{[p^s]}+p^rD}\cong\frac{D}{(\alpha^{[p^{s-r+1}]}\cap D)+pD}$
	\end{claim}
	\begin{claimproof}
		We note that $\alpha^{[p^s]}\cap D$ is the same as calculating $\alpha^{[p^s]}$ within $D$. Now, we have:
		\begin{align*}
			\frac{\alpha^{[p^s]}+p^{r-1}D}{\alpha^{[p^s]}+p^rD}\cong &\frac{p^{r-1}D}{p^{r-1}D\cap(\alpha^{[p^s]}+p^rD)} \\
			=& \frac{p^{r-1}D}{p^{r-1}D\cap (\alpha^{[p^s]}\cap D)+p^rD} \\
			=&\frac{p^{r-1}D}{p^{r-1}(\alpha^{[p^{s-r+1}]}\cap D)+p^rD} & \tx{by \Cref{keylemma}} \\
			\cong& \frac{D}{(\alpha^{[p^{s-r+1}]}\cap D)+pD}
		\end{align*}
		
	\end{claimproof}
	\noindent To conclude, note that $(\forall D> \beta > \alpha)((\alpha^{[p^{s-r+1}]}\cap D)+pG\subseteq \beta+pG)$. Combined with $|\{\beta+pG :\beta\in\mathcal{S}_p\wedge D>\beta>\alpha\}|=\infty$, we get $[D:(\alpha^{[p^{s-r+1}]}\cap D)+pD]=\infty$. Observe that $(\alpha^{[p^{s}]}\cap D)+pD\leq (\alpha^{[p^{s-r+1}]}\cap D)+pD$, and so $[G:(\alpha^{[p^s]}\cap D)+pG]\geq[G:(\alpha^{[p^{s-r+1}]}\cap D)+pG]=\infty$.
\end{proof}

\begin{lemma}\label{leftnonconvexconvexrelationship}
	Suppose that $H=\alpha^{[p^s]}$ with $\alpha\leq G$ convex, $H\neq D+p^sG$ for any $D\leq G$ convex, and $K\leq G$ is convex with $\alpha <K\leq G$. Then for $r\leq s$, $$[K+p^rG:\alpha^{[p^s]}+p^rG] =\infty = [\alpha^{[p^s]}+K\cap p^{r-1}G+p^rG:\alpha^{[p^s]}+p^rG] $$
\end{lemma}
\begin{proof}
	First, note that 
	\begin{align*}
		\frac{K+p^rG}{\alpha^{[p^s]}+p^rG} \cong& \frac{K}{K\cap(\alpha^{[p^s]}+p^rG)} 
	\end{align*}
	Now, suppose $x = a+p^rg_1$ with $a\in \alpha^{[p^s]}, g_1\in G$. Then by the definition of $\alpha^{[p^s]}$, we have that $\alpha^{[p^s]}\subseteq K+p^sG$. So $a = k + p^sg_2$ with $k\in K, g_2\in G$. Note that as $\alpha^{[p^s]} = \alpha^{[p^s]}+p^sG$, we have $k\in \alpha^{[p^s]}$. Now, as $s\geq r$, we have $x\in K\cap\alpha^{[p^s]}+p^rG$. Hence, $K\cap (\alpha^{[p^s]}+p^rG)\subseteq K\cap(K\cap\alpha^{[p^s]}+p^rG)$. The reverse inclusion is clear. Hence, we have that 
	\begin{align*}
		\frac{K}{K\cap(\alpha^{[p^s]}+p^rG)} = & \frac{K}{K\cap(K\cap \alpha^{[p^s]}+p^rG)} \\
		= & \frac{K}{K\cap \alpha^{[p^s]}+K\cap p^rG} \\
		= & \frac{K}{K\cap \alpha^{[p^s]}+p^rK}
	\end{align*}
	Note also that
	\begin{align*}
		\frac{\alpha^{[p^s]}+K\cap p^{r-1}G +p^rG}{\alpha^{[p^s]}+p^rG} \cong& \frac{K\cap p^{r-1}G}{K\cap p^{r-1}G\cap(\alpha^{[p^s]}+p^rG)} \\
		\cong& \frac{p^{r-1}K}{p^{r-1}G\cap(K\cap\alpha^{[p^s]}+p^rG\cap K)} \\
		= & \frac{p^{r-1}K}{K\cap\alpha^{[p^s]}\cap p^{r-1}G+p^rK} \\
		= & \frac{p^{r-1}K}{K\cap p^{r-1}\alpha^{[p^{s-r+1}]}+p^rK} \\
		= & \frac{p^{r-1}K}{p^{r-1}(K\cap \alpha^{[p^{s-r+1}]}+pK)} \\
		\cong & \frac{K}{K\cap \alpha^{[p^{s-r+1}]}+pK}
	\end{align*}
	Now, note that $K$ is an oag with $\alpha\leq K$ convex, and $\alpha^{[p^s]}\cap K$ is the same as evaluating $\alpha^{[p^s]}$ within $K$. Hence, we may apply \Cref{infiniterightquotientsnonconvex} and conclude that both of these quotients are infinite.
\end{proof}
\begin{corollary}\label{[b^p^s :a^p^s] infinite}
	Suppose that $H = \alpha^{[p^s]}, K = \beta^{[p^t]}$ with $\alpha < \beta \leq G$ convex, and $\alpha^{[p^s]}\neq D+p^sG$ for any $D\leq G$ convex. Then for $r\leq s$, $$[K+p^rG:\alpha^{[p^s]}+p^rG] =\infty = [\alpha^{[p^s]}+K\cap p^{r-1}G+p^rG:\alpha^{[p^s]}+p^rG] $$
\end{corollary}
\begin{proof}
	Immediate after noticing that $\beta\leq K$, and hence$$[K+p^rG:\alpha^{[p^s]}+p^rG]\geq [\beta+p^rG:\alpha^{[p^s]}+p^rG]$$ and $$[\alpha^{[p^s]}+K\cap p^{r-1}G+p^rG:\alpha^{[p^s]}+p^rG]\geq [\alpha^{[p^s]}+\beta \cap p^{r-1}G+p^rG:\alpha^{[p^s]}+p^rG]$$ and then applying \Cref{leftnonconvexconvexrelationship}.
\end{proof}

\begin{lemma}\label{Describing a^p^j - a^p^s quotients}
	Let $0\leq i < r<s$. Then
	$$\frac{\alpha^{[p^{s-1}]}\cap p^iG + \alpha^{[p^s]}+p^rG}{\alpha^{[p^{s-1}]}\cap p^{i+1}G + \alpha^{[p^s]}+p^rG} \cong \frac{\alpha^{[p^{s-i-1}]}+pG}{\alpha^{[p^{s-i}]}+pG}$$
	And hence we have
	$$[\alpha^{[p^{s-1}]}+p^rG:\alpha^{[p^s]}+p^rG] = \prod\limits_{1\leq i 
		\leq r} [{\alpha^{[p^{s-i}]}+pG}:{\alpha^{[p^{s-i+1}]}+pG}]$$
	And so $$[\alpha^{[p^{s-1}]}+p^iG:\alpha^{[p^s]}+p^iG]\leq [\alpha^{[p^{s-1}]}+p^rG:\alpha^{[p^s]}+p^rG]$$
\end{lemma}
\begin{proof}
	The second and third statements are immediate from the first, and the first is a straightforward application of \Cref{keylemma} and group isomorphism theorems.
\end{proof}
\noindent With the following Lemmas (3.2 and 3.3 from \cite{QE}) we now have enough to prove \Cref{(theorem) main for a^p^s+p^rG}:
\begin{lemma}\label{QE 3.2}
    Suppose we have an abelian group $G$, a subgroup $G'\leq G$ and a subset $Y\subseteq G$ of the form
    \begin{align*}
            Y = (a_0+M_0)\setminus \bigcup\limits_{1\leq i\leq k}(a_i+M_i)
    \end{align*}
    where the $M_i$ are subgroups of $G$, $a_i\in G$, and $a_i+M_i\subseteq a_0+M_0$ for $1\leq i \leq j$, and $(a_i+M_i)\cap (a_j+M_j)=\emptyset$ for $1\leq i < j\leq k$. Then for $y\in G$ we have that $y\in Y'=Y+G'$ iff both of the following hold:
    \begin{align}
    	y-a_0 & \in M_0 + G' \\
    	\sum\limits_{\substack{1\leq i\leq k \\ y-a_a\in M_i+G'}} & \frac{1}{[M_0\cap G':M_i\cap G'] } < 1
    \end{align}
\end{lemma}
\begin{lemma}\label{QE 3.3}
    Suppose that $n\in\mathbb{N}, n\geq 2$, and that $q_1,\dots,q_k$ are powers of $n$. Then there exists an $N\in\mathbb{N}$ depending only on $n$ and $k$ such that $$\sum\limits_{1\leq i\leq k} q_i^{-1} \geq 1 \iff \sum\limits_{\substack{1\leq i\leq k \\ q_i < n^N}} q_i^{-1} \geq 1$$
\end{lemma}

\begin{theorem}\label{main for a^p^s+p^rG}
	Let $d,e\in\mathbb{N}_{>0}$, $B\subseteq G$, $\mathcal{A}$ a $B$-definable subset of the spines, $1\leq r \leq s$, and $\left\{f_\alpha:G^d\rightarrow (\frac{G}{\alpha^{[p^s]}+p^rG})^{(e)}\right\}_{\alpha\in\mathcal{A}}$ be a $B$-definable family of functions. Then we can find:
	\begin{itemize}
        \item $k,\ell\in\mathbb{N}$, and $n_h, m_h\in \mathbb{N}$ for $h\leq k$
        \item $B$-definable subsets of the spines $\mathcal{A}_1,\dots,\mathcal{A}_\ell$  which partition $\mathcal{A}$
		\item $B$-definable families $X_1^\alpha,\dots,X_k^\alpha\subseteq G^d$ which partition $G^d$ for every fixed $\alpha\in\mathcal{A}$
		\item $1\leq l_{i,h,1} <\dots<l_{i,h,n_h}\leq s$, $r = j_{i,h,1}>\dots>j_{i,h,n_h}\geq 1$ for $i\leq\ell, h\leq k$ such that for all $\alpha\in\mathcal{A}_i$,
        \begin{align*}
            \left[\bigcap\limits_{1\leq t\leq n_h}(\alpha^{[p^{l_{i,h,t}}]}+p^{j_{i,h,t}}G):\alpha^{[p^s]}+p^rG\right]<\infty
        \end{align*}
		\item $B$-definable linear functions $f_{i,h,t}:G^d\rightarrow G$ for $1\leq t\leq n_h$, $1\leq i\leq\ell$, $1\leq h\leq k$
		\item $1\leq v_{i,h,1},\dots,v_{i,h,m_h}\leq s$, $1\leq w_{i,h,1},\dots,w_{i,h,m_h}\leq r$
		\item $B$-definable linear functions $g_{i,h,t}:G^d\rightarrow G$ for $1\leq t\leq m_h$, $1\leq i\leq\ell$, $1\leq h\leq k$
	\end{itemize} such that for $\alpha \in\mathcal{A}_i$, $f_\alpha = \bigcup\limits_{1\leq h\leq k} ((\bigcap\limits_{t\leq n_h}f_{i,h,t}^{H_{i,h,t}}|_{X_h^\alpha})\setminus(\bigcup\limits_{t\leq m_h}g_{i,h,t}^{K_{i,h,t}}|_{X_h^\alpha}))$, where $H_{i,h,t} = \alpha^{[p^{l_{i,h,t}}]}+p^{j_{i,h,t}}G$ and $K_{i,h,t} = \alpha^{[p^{v_{i,h,t}}]}+p^{w_{i,h,t}}G$. 
\end{theorem}
\begin{proof}
	By \cite{QE}, we may assume that $f_\alpha$ is defined by a formula of the form
	\begin{align*}
		\bigvee\limits_{i}(\exists\bar\theta)(\xi_i(\bar\theta, \bar\eta,\alpha)\wedge \psi_i(\bar x, y, \bar\theta))
	\end{align*}
	Where $\xi_i$ is a formula restricted to an $n$-spine, and $\psi_i$ is a conjunction of literals, with each atom of the form $r_iy\square_\alpha t_i(\bar x)$ where $t_i(\bar x)$ is a linear term and $\square \in \{<, \leq, >, \geq, \equiv_n, \equiv_n^{[m]}\}$.
    For now, we work with a fixed $\alpha\in\mathcal{A}$, and set $f = f_\alpha$; we suppress the $\alpha$ appearing in the definition and absorb it into $\bar\eta$.
	\nl\nl
	Note that by restricting to $X_i$ defined by $(\exists y)(\exists\bar\theta)(\xi_i(\bar\theta, \bar\eta)\wedge \psi_i(\bar x, y, \bar\theta))$ we may further reduce to
	\begin{align*}
		(\exists z)(y\equiv_{p^r}z\wedge(\exists\bar\theta)(\xi(\bar\theta, \bar\eta)\wedge \psi(\bar x, z, \bar\theta)))
	\end{align*}
	And furthermore, by following the quantifier-elimination proof in \cite{QE} and making use of the Chinese Remainder Theorem, we may reduce this to 
	\begin{align*}
		(\exists\bar\theta)(\xi(\bar\theta, \bar\eta)\wedge \chi(\bar x, y, \bar\theta))
	\end{align*}
	Where $\chi$ is a conjunction of literals with atoms of the form $r_iy\diamond_\beta t_i(\bar x)$, with $\diamond \in \{\equiv_{p^j}, \equiv_{p^j}^{[p^k]}\}$ and $1\leq j\leq r,k$. We may also assume that $\models (\forall \bar x)(\forall\bar\theta)((x\in\tx{dom}(f)\wedge\xi(\bar\theta,\eta))\rightarrow(\exists^{\geq 1}y)(\chi(\bar x, y,\bar\theta)$.
    \nl\nl
    The following claim lets us simplify the formula further:
    \begin{claim}
        We may also assume that $\beta \geq \alpha$, and if $\beta = \alpha$ then $k\leq s$
    \end{claim}
    \begin{claimproof}
        For fixed $\bar\theta$, we we have a consistent conjunction of congruence conditions coming from a family of lattices of subgroups of the following form:

        
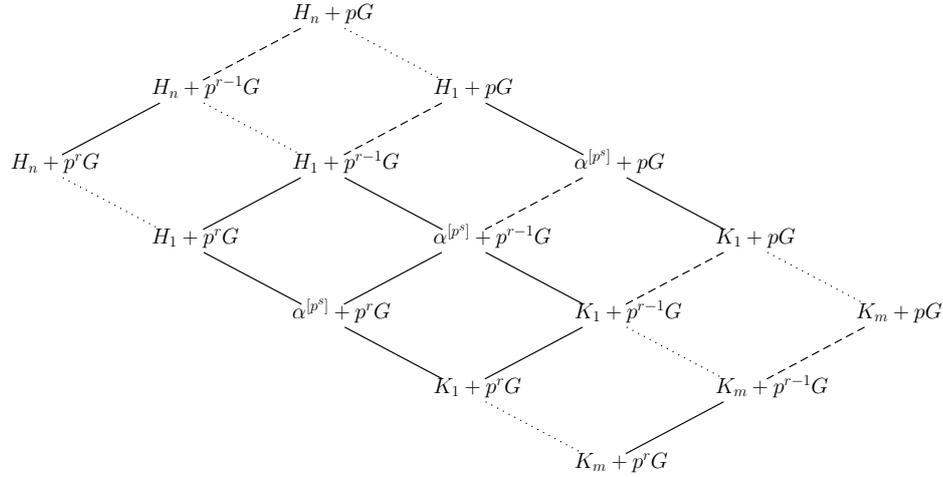
\[\begin{tikzcd}[
        	scale cd=.4,
            font = \huge,
        	text height = height("$H_1+p^{r-1}G$"),
        	text width = width("$H_1+p^{r-1}G$"),
        	align = center
        	]
        	&& {H_n+pG} \\
        	& {H_n+p^{r-1}G} && {H_1+pG} \\
        	{H_n+p^rG} && {H_1+p^{r-1}G} && {\alpha^{[p^s]}+pG} \\
        	& {H_1+p^rG} && {\alpha^{[p^s]}+p^{r-1}G} && {K_1+pG} \\
        	&& {\alpha^{[p^s]}+p^rG} && {K_1+p^{r-1}G} && {K_m+pG} \\
        	&&& {K_1+p^rG} && {K_m+p^{r-1}G} \\
        	&&&& {K_m+p^rG}
        	\arrow[dashed, no head, from=1-3, to=2-2]
        	\arrow[dotted, no head, from=1-3, to=2-4]
        	\arrow[no head, from=2-2, to=3-1]
        	\arrow[dotted, no head, from=2-2, to=3-3]
        	\arrow[dashed, no head, from=2-4, to=3-3]
        	\arrow[no head, from=2-4, to=3-5]
        	\arrow[dotted, no head, from=3-1, to=4-2]
        	\arrow[no head, from=3-3, to=4-2]
        	\arrow[no head, from=3-3, to=4-4]
        	\arrow[dashed, no head, from=3-5, to=4-4]
        	\arrow[no head, from=3-5, to=4-6]
        	\arrow[no head, from=4-2, to=5-3]
        	\arrow[no head, from=4-4, to=5-3]
        	\arrow[no head, from=4-4, to=5-5]
        	\arrow[dashed, no head, from=4-6, to=5-5]
        	\arrow[dotted, no head, from=4-6, to=5-7]
        	\arrow[no head, from=5-3, to=6-4]
        	\arrow[no head, from=5-5, to=6-4]
        	\arrow[dotted, no head, from=5-5, to=6-6]
        	\arrow[dashed, no head, from=5-7, to=6-6]
        	\arrow[dotted, no head, from=6-4, to=7-5]
        	\arrow[no head, from=6-6, to=7-5]
        \end{tikzcd}\]
        Where $K_i, H_i$ are either convex or of the form $\beta^{[p^k]}$. For $K_i$ we have $\beta \leq\alpha$, and if $\beta = \alpha$ then $k> s$; on the other hand for $H_i$ we have that $\beta \geq\alpha$, and if $\alpha = \beta$ then $k< s$. As an intermediate step it is helpful to modify the congruence conditions so that each individual condition either excludes or restricts to exactly 1 coset. We may do at the cost of having the terms potentially be $\mathbb{Z}[\frac{1}{p}]$-linear; we do however still retain $\mathbb{Z}$-linearity for the $p^rG$-congruences. Specifically, given a congruence of the form $my\equiv_{K_m,p^k} t(\bar x) + g$ where $p|m$, we may replace it by $\frac{m}{p}y\equiv_{K_m,p^{k-1}} \frac{1}{p}(t(\bar x) + g)$, and we may repeat this until $p\nmid m$.
        \nl\nl
        We show by induction on $m$ that we may uniformly remove and replace the congruence conditions coming from $K_i+p^jG$ with $\alpha^{[p^s]}+p^\ell G$ congruence conditions. In each step of the induction, we split the above lattice into two pieces as follows:
        
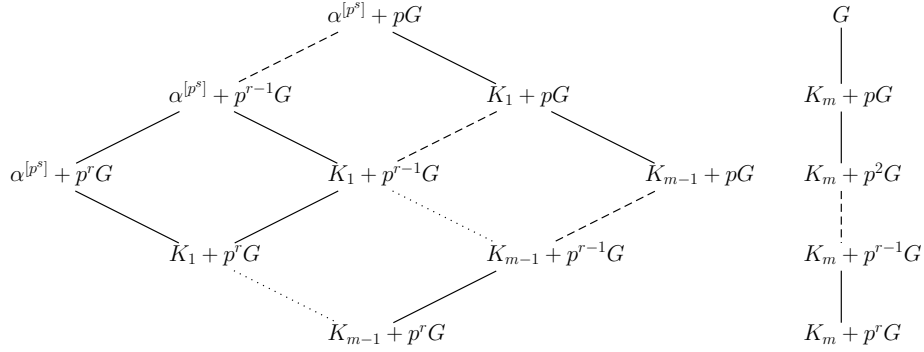
\[\begin{tikzcd}[
        	scale cd=.5,
            font = \LARGE,
        	text height = height("$H_1+p^{r-1}G$"),
        	text width = width("$H_1+p^{r-1}G$"),
        	align = center
        	]
        	&& {\alpha^{[p^s]}+pG} &&& G \\
        	& {\alpha^{[p^s]}+p^{r-1}G} && {K_1+pG} && {K_m+pG} \\
        	{\alpha^{[p^s]}+p^rG} && {K_1+p^{r-1}G} && {K_{m-1}+pG} & {K_m+p^2G} \\
        	& {K_1+p^rG} && {K_{m-1}+p^{r-1}G} && {K_m+p^{r-1}G} \\
        	&& {K_{m-1}+p^rG} &&& {K_m+p^rG}
        	\arrow[dashed, no head, from=1-3, to=2-2]
        	\arrow[no head, from=1-3, to=2-4]
        	\arrow[no head, from=1-6, to=2-6]
        	\arrow[no head, from=2-2, to=3-1]
        	\arrow[no head, from=2-2, to=3-3]
        	\arrow[dashed, no head, from=2-4, to=3-3]
        	\arrow[no head, from=2-4, to=3-5]
        	\arrow[no head, from=2-6, to=3-6]
        	\arrow[no head, from=3-1, to=4-2]
        	\arrow[no head, from=3-3, to=4-2]
        	\arrow[dotted, no head, from=3-3, to=4-4]
        	\arrow[dashed, no head, from=3-5, to=4-4]
        	\arrow[dashed, no head, from=3-6, to=4-6]
        	\arrow[dotted, no head, from=4-2, to=5-3]
        	\arrow[no head, from=4-4, to=5-3]
        	\arrow[no head, from=4-6, to=5-6]
        \end{tikzcd}\]
        with $K_{m-1}+p^rG\leq \alpha^{[p^s]}+p^rG$. Note that we may assume that the tower on the right contains at most $1$ positive congruence relation, and also that the negative congruence relations are disjoint (since they are either disjoint or contained in one another, and this is a definable condition). So, if we set the congruence conditions from the right tower give rise to a subset of the form
        \begin{align*}
            Y = (a_0+M_0)\setminus \bigcup\limits_{1\leq i\leq k}(a_i+M_i)
        \end{align*}
        where $M_i = K_m+p^{r_i}G$ for some $0\leq r_i\leq r$. Then, noting that what remains on the left is a union of $K_{m-1}+p^rG$-cosets, we may replace $Y$ by $Y+K_{m-1}+p^rG=Y+K_{m-1}$. Hence, we are in a position to make use of \Cref{QE 3.2} with $G'=K_{m-1}+p^rG$, so $y\in X+K_{m-1}+p^rG$ iff both of the following hold:
        \begin{align}
            \setcounter{equation}{0}
        	y-a_0&\in M_0 + K_{m-1} +p^rG \\
        	\sum\limits_{\substack{1\leq i\leq k \\ y-a_i\in M_i+K_{m-1}+p^rG}} &\frac{1}{[M_0\cap (K_{m-1}+p^rG):M_i\cap (K_{m-1}+p^rG)] } < 1
        \end{align}
        Now, if we have no positive congruence relations, (1) is equivalent to $y\in G$, and otherwise it is equivalent to $my \equiv_{K_{m-1}, p^{r_0}} t(\bar x) + g$. We now need only show that (2) is definable. To do this, first note that $[M_0\cap (K_{m-1}\cap p^rG):M_i\cap (K_m+p^rG)]$ is either infinite or a power of $p$, since these are both intermediate subgroups between $G$ and $p^rG$, and $\frac{G}{p^rG}$ is a $p$-group, and any quotient of subgroups of a $p$-group is again a $p$-group. Then by \Cref{QE 3.3}, (2) is equivalent to a finite boolean combination of formulas of the form $[M_0\cap (K_{m-1}+p^rG):M_i\cap (K_{m-1}+p^rG)] = p^j$, and is hence also definable, and can be added as a case distinction fixing the set that the sum in (2) runs over.
    \end{claimproof}
    \noindent Now, for fixed $\theta$, this gives rise to congruence conditions coming from a lattice of the following form, where $k<s$ is maximal such that a congruence of the form $r_iy\equiv_{\alpha,{p^j}}^{[p^k]} t_i(\bar x)$ appears, $h<s$ is minimal such that such a congruence appears, $\beta>\alpha$ is minimal such that a congruence involving $\beta$ appears, and $\gamma>\alpha$ is maximal such that a congruence involving $\gamma$ appears:
	
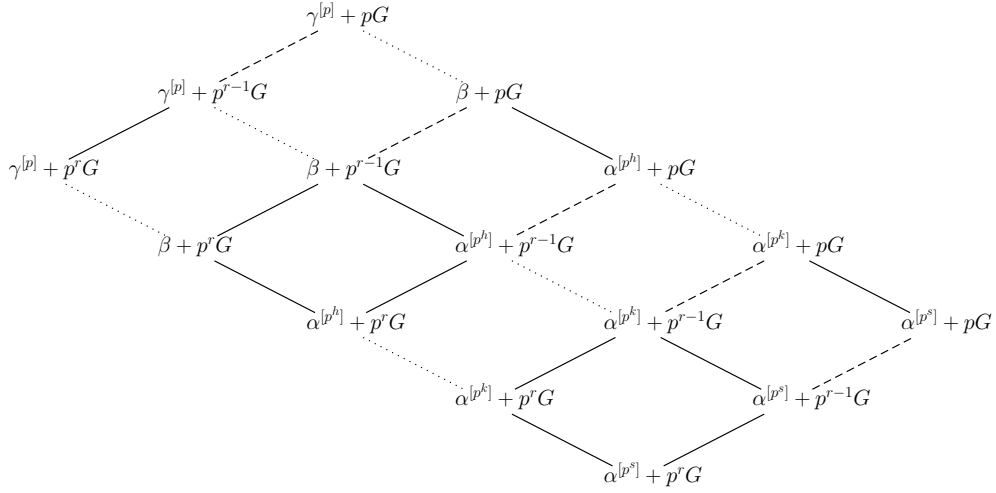
\[\begin{tikzcd}[
		scale cd=.4,
        font = \huge,
		text height = height("$\alpha^{[p^k]}+p^{r-1}G$"),
		text width = width("$\alpha^{[p^k]}+p^{r-1}G$"),
		align = center
		]
		&& {\gamma^{[p]}+pG} \\
		& {\gamma^{[p]}+p^{r-1}G} && {\beta+pG} \\
		{\gamma^{[p]}+p^rG} && {\beta+p^{r-1}G} && {\alpha^{[p^h]}+pG} \\
		& {\beta+p^rG} && {\alpha^{[p^h]}+p^{r-1}G} && {\alpha^{[p^k]}+pG} \\
		&& {\alpha^{[p^h]}+p^rG} && {\alpha^{[p^k]}+p^{r-1}G} && {\alpha^{[p^s]}+pG} \\
		&&& {\alpha^{[p^k]}+p^rG} && {\alpha^{[p^s]}+p^{r-1}G} \\
		&&&& {\alpha^{[p^s]}+p^rG}
		\arrow[dashed, no head, from=1-3, to=2-2]
		\arrow[dotted, no head, from=1-3, to=2-4]
		\arrow[no head, from=2-2, to=3-1]
		\arrow[dotted, no head, from=2-2, to=3-3]
		\arrow[dashed, no head, from=2-4, to=3-3]
		\arrow[no head, from=2-4, to=3-5]
		\arrow[dotted, no head, from=3-1, to=4-2]
		\arrow[no head, from=3-3, to=4-2]
		\arrow[no head, from=3-3, to=4-4]
		\arrow[dashed, no head, from=3-5, to=4-4]
		\arrow[dotted, no head, from=3-5, to=4-6]
		\arrow[no head, from=4-2, to=5-3]
		\arrow[no head, from=4-4, to=5-3]
		\arrow[dotted, no head, from=4-4, to=5-5]
		\arrow[dashed, no head, from=4-6, to=5-5]
		\arrow[no head, from=4-6, to=5-7]
		\arrow[dotted, no head, from=5-3, to=6-4]
		\arrow[no head, from=5-5, to=6-4]
		\arrow[no head, from=5-5, to=6-6]
		\arrow[dashed, no head, from=6-6, to=5-7]
		\arrow[no head, from=7-5, to=6-4]
		\arrow[no head, from=7-5, to=6-6]
	\end{tikzcd}\]
	First, suppose that no congruence from a subgroup of the form $\alpha^{[p^\ell]}+p^jG$ appears for $\ell < s$. Then, we have the following arrangement of subgroups:

    
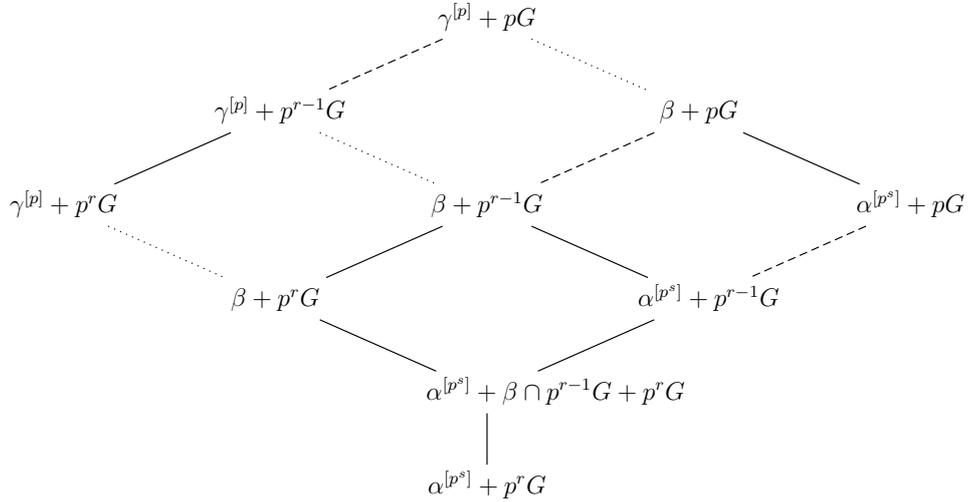
\[\begin{tikzcd}[
			scale cd=.8,
            font = \large,
			text height = height("$H_1+p^{r-1}G$"),
			text width = width("$H_1+p^{r-1}G$"),
			align = center
		]
		&& {\gamma^{[p]}+pG} \\
		& {\gamma^{[p]}+p^{r-1}G} && {\beta+pG} \\
		{\gamma^{[p]}+p^rG} && {\beta+p^{r-1}G} && \alpha^{[p^s]}+pG \\
		& {\beta+p^rG} && {\alpha^{[p^s]}+p^{r-1}G} \\
		&& {\alpha^{[p^s]}+\beta\cap p^{r-1}G + p^rG} \\
		&& {\alpha^{[p^s]}+p^rG}
		\arrow[dashed, no head, from=1-3, to=2-2]
		\arrow[dotted, no head, from=1-3, to=2-4]
		\arrow[no head, from=2-2, to=3-1]
		\arrow[dotted, no head, from=2-2, to=3-3]
		\arrow[dashed, no head, from=2-4, to=3-3]
		\arrow[no head, from=2-4, to=3-5]
		\arrow[dotted, no head, from=3-1, to=4-2]
		\arrow[no head, from=3-3, to=4-2]
		\arrow[no head, from=3-3, to=4-4]
		\arrow[dashed, no head, from=3-5, to=4-4]
		\arrow[no head, from=4-2, to=5-3]
		\arrow[no head, from=4-4, to=5-3]
		\arrow[no head, from=5-3, to=6-3]
	\end{tikzcd}\]
    But, by \Cref{[b^p^s :a^p^s] infinite}, $[\alpha^{[p^s]}+\beta\cap p^{r-1}G + p^rG:\alpha^{[p^s]}+p^rG]=\infty$. If there is no positive congruence of the form $ry\equiv_{\alpha, p^r}^{[p^s]}t(\bar x)$ appearing, this conjunction of congruence conditions results in at least one coset of $\alpha^{[p^s]}+\beta\cap p^{r-1}G + p^rG$ with only a finite number of $\alpha^{[p^s]}+p^rG$ cosets being excluded - which results in an infinite number of $\alpha^{[p^s]}+p^rG$ cosets. Hence, there must be such a positive congruence condition, and $f$ must be given by this single linear congruence.
    \nl\nl
    We can now simplify the setup further; consider the function $f_h:G^d\rightarrow \frac{G}{\alpha^{[p^h]}+p^rG}$ given by $f_h(x) = f(x) + \alpha^{[p^h]}$. By following the process for removing congruences on smaller subgroups, we see that this function is piecewise given by single congruence conditions $ry\equiv_{\alpha,p^r}^{[p^h]}t(\bar x)$ where $ry\equiv_{\alpha,p^r}^{[p^j]}t(\bar x)$ is a congruence appearing in the definition of $f$ for some $j\geq h$.
    \nl\nl
    Hence we may remove all congruences relating to some $\beta > \alpha$, since all of these congruences must be implied by $ry\equiv_{\alpha,p^r}^{[p^j]}t(\bar x)$. Immediately we see that $f$ is therefore piecewise defined by purely a conjunctions of literals with atoms $ry\diamond_\alpha t(\bar x)$ with $\diamond\in{\equiv_{p^j},\equiv_{p^j}^{[p^i]}}$ and $1\leq j\leq r, j\leq i\leq s$; there is no longer a family-union we need to consider, and we may take the defining formula to be of the form $\chi(\bar x, y)$
    \nl\nl
	Now, let $\ell\geq 1$ be maximal such that:
	\begin{enumerate}[(1)]
		\item There is a congruence of the form $ry\equiv_{\alpha,{p^j}}^{[p^\ell]} t(\bar x)$ appearing
		\item $[\alpha^{[p^\ell]}+p^rG:\alpha^{[p^s]}+p^rG] = \infty$
	\end{enumerate}
    If there is no such $\ell$, then the result is immediate by once again considering $f_h(x) = f(x) + \alpha^{[p^h]}$ and repeating the previous observation that there is a positive congruence coming from some $\alpha^{[p^m]}+p^rG$.
	\nl\nl
	Otherwise, let $i$ be minimal such that 
    \begin{align*}
        \left[(\alpha^{[p^s]}+p^iG)\cap(\alpha^{[p^\ell]} + p^r G): \alpha^{[p^s]}+p^rG\right] <\infty
    \end{align*}

	\begin{claim}
		There are some $j,m$ with $r \geq j \geq i$, $\ell > m \geq s$ and a positive congruence condition $C(\bar x,y):=$\cs$ ry\equiv_{\alpha,{p^j}}^{[p^m]} t(\bar x)$\ce , such that $C(\bar x, y)$ is a literal appearing as a conjunct of $\chi(\bar x, y)$
	\end{claim}
	\begin{claimproof}
		We will show that if the claim does not hold, then a $\chi(\bar x, y)$ is in fact an infinite union of $\alpha^{[p^s]}+p^rG$ cosets.
		\nl\nl
		Suppose no such congruence condition exists. Then, any congruence condition of the form $ry\equiv_{\alpha,{p^j}}^{[p^m]} t(\bar x)$ with $j,m$ as above must appear as a negative congruence condition; note that we may assume at least one such exists, as this only decreases the number of cosets we will end up with. Next, we may replace all instances of $\equiv_{\alpha, p^j}^{[p^m]}$ with finitely many instances of $\equiv_{\alpha, p^j}^{[p^s]}$, since by choice of $l$ in combination with \Cref{Describing a^p^j - a^p^s quotients} we have that $[\alpha^{[p^m]}+p^iG:\alpha^{[p^s]}+p^iG] <\infty$. Now, we may replace instances of $\equiv_{\alpha, p^j}^{[p^s]}$ with $\equiv_{\alpha, p^i}^{[p^s]}$, since this also only decreases the number of resulting cosets.
		\nl\nl
		Now, we are in a situation where we have some union of $(\alpha^{[p^l]}+p^rG)\cap(\alpha^{[p^s]}+p^{i-1}G)$ cosets intersected with a cofinite number of cosets of $\alpha^{[p^s]}+p^iG$. But, by choice of $i$, we have that 
        \begin{align*}
		    \left[(\alpha^{[p^l]}+p^rG)\cap(\alpha^{[p^s]}+p^{i-1}G):(\alpha^{[p^l]}+p^rG)\cap(\alpha^{[p^s]}+p^iG)\right]=\infty
		\end{align*}
		And so this intersections results in infinitely many $(\alpha^{[p^l]}+p^rG)\cap(\alpha^{[p^s]}+p^iG)$ cosets, and hence infinitely many $\alpha^{p^s}+p^rG$ cosets.
	\end{claimproof}
	\noindent Let $j,m$ be a witness to the above claim with $j$ maximal, with $C(\bar x, y)$ as in the claim. We may write $\chi(\bar x, y)=\psi(\bar x, y)\wedge \phi(\bar x, y)$, where $\phi(\bar x, y)$ contains all congruences relating to $\alpha^{[p^t]}$ with $l< t\leq s$, and $\psi(\bar x, y, \bar\theta)$ contains all remaining congruences. Now, $\chi(\bar x, G) + \alpha^{[p^l]}+p^rG$ defines a function $g:G^d\rightarrow (\frac{G}{\alpha^{[p^l]}+p^rG})^{(n)}$ for some $n\leq e$, and so by induction we may find
	\begin{itemize}
		\item $B$-definable sets $\{X_h\}_{h\leq k}$
		\item $1\leq l_{1},\dots,l_{n},v_{1},\dots,v_{m}\leq l$
		\item $1\leq j_{1},\dots,j_{n},w_{1},\dots,w_{m}\leq r$
		\item $f_{t}:G^d\rightarrow G$ for $1\leq t\leq n$
		\item $g_{t}:G^d\rightarrow G$ for $1\leq t\leq m$
	\end{itemize}
	as in the theorem statement (with $\mathcal{A} = \mathcal{A}_1$ and $k=1$) such that $g = \bigcup\limits_{h\leq k} ((\bigcap\limits_{t\leq n}f^{H_{t}}_{t}|_{X_1})\setminus(\bigcup\limits_{t\leq m}g^{K_{t}}_{t}|_{X_1}))$ and $\left[\bigcap\limits_{1\leq t\leq n}(\alpha^{[p^{l_{t}}]}+p^{i_{t}}G):\alpha^{p^l}+p^rG\right]<\infty$.
	\nl\nl
	Furthermore, we note that $\chi(\bar x, G) + \alpha^{[p^s]}+p^rG = \psi(\bar x, G)\cap (\phi(\bar x, y) + \alpha^{[p^l]})$, and in particular we conclude that $\chi(\bar x, y) = g(\bar x)\cap\phi(\bar x, y)$, and so upon restricting to $X_h$ we find that $\chi(\bar x, y, \bar\theta) = g_h(\bar x, y)\cap\phi(\bar x, y)$, where $g_h = (\bigcap\limits_{t\leq n_h}f_{h,t}|_{X_h})\setminus(\bigcup\limits_{t\leq m_h}g_{h,t}|_{X_h})$.
    \nl\nl
	Also, note that $\phi(\bar x, y)$ contains $C(\bar x, y)$, and so the intersection of all of the subgroups containing a positive congruence is contained in $\bigcap\limits_{1\leq t\leq n_j}(
    \alpha^{[p^{l_{j,t}}]}+p^{i_{j,t}})\cap (\alpha^{[p^m]}+p^jG)$, which is a finite index extension of $\alpha^{[p^s]}$. And so taking $X_h$ given by the earlier case distinctions and removal of the first disjunction we arrive at the theorem in the case of fixed $\alpha$.
    \nl\nl
    All that remains is to generalise this from fixed $\alpha$ to a family $\mathcal{A}$. Note, however, that regardless of the indexes of $\alpha^{[p^m]}+p^jG$ over $\alpha^{[p^s]}+p^jG$, there are only finitely many possible choices of $f_{h,t},g_{h,t}$ (coming from terms appearing in the original definition), as well as choices for $l_{h,t},j_{h,t},v_{h,t},w_{h,t}$. The case distinctions are also clearly uniformly definable from $\alpha$, and so we may partition $\mathcal{A}$ into definable subsets based on which of these choices and which families of case distinctions give rise to $g_\alpha$ with $g_\alpha \subseteq f_\alpha$. Note that the finite index requirement is automatic for any configuration that gives a subset of $f$; if the index of the intersections from positive congruences is not finite then we do not obtain $g_\alpha\subseteq f_\alpha$, as seen earlier in the proof.
	
\end{proof}

\section{Functions to quotients by subgroups of the form $D+p^rG$ with $D\leq G$ a definable convex subgroup}\label{D+p^rG section}

The goal of this section is to prove \Cref{main} in the case where $H =
D + p^rG$ with $D$ any convex definable subgroup of G. It will be helpful to prove a lot of supplementary facts regarding the interactions between various definable subgroups, each of which contains $D+p^rG$. For notational simplicity we prove these facts for $p^rG$, noting that they are immediately applicable to $D+p^rG$ as well since $G/D$ is a convex subgroup, and $D+p^rG$ is the lift of $p^r(G/D)$. To start with, it will be helpful to consider how the following lattice of subgroups behaves:
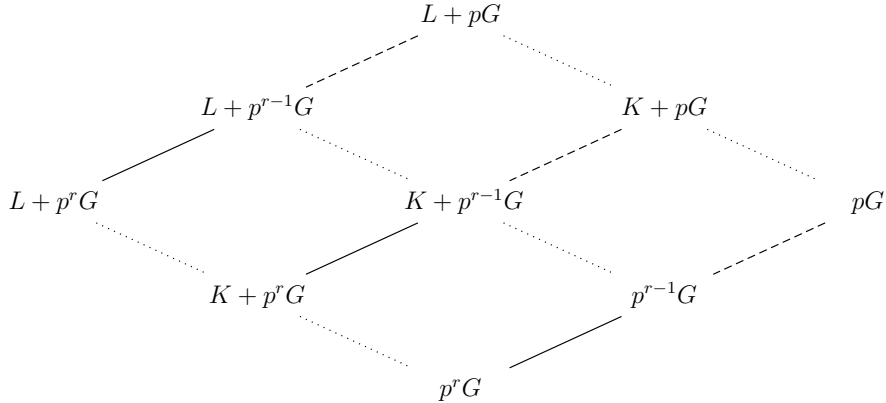
\[\begin{tikzcd}[
	scale cd = .8,
    font = \large,
	text height = height("$K+p^{r-1}G$"),
	text width = width("$K+p^{r-1}G$"),
	align = center
	]
	&& {L+pG} \\
	& {L+p^{r-1}G} && {K+pG} \\
	{L+p^rG} && {K+p^{r-1}G} && {pG} \\
	& {K+p^rG} && {p^{r-1}G} \\
	&& {p^rG}
	\arrow[dashed, no head, from=1-3, to=2-2]
	\arrow[dotted, no head, from=1-3, to=2-4]
	\arrow[no head, from=2-2, to=3-1]
	\arrow[dotted, no head, from=2-2, to=3-3]
	\arrow[dashed, no head, from=2-4, to=3-3]
	\arrow[dotted, no head, from=2-4, to=3-5]
	\arrow[dotted, no head, from=3-1, to=4-2]
	\arrow[no head, from=3-3, to=4-2]
	\arrow[dotted, no head, from=3-3, to=4-4]
	\arrow[dashed, no head, from=3-5, to=4-4]
	\arrow[dotted, no head, from=4-2, to=5-3]
	\arrow[no head, from=5-3, to=4-4]
\end{tikzcd}\]
Where $K,L$ are of the form $\alpha$ or $\alpha^{[p^s]}$, noting that for $\alpha < \beta$ and $s,k\in\mathbb{N}$ we have 
\begin{align*}
    \alpha+p^kG \subseteq \alpha^{[p^s]}+p^kG\subseteq\alpha^{[p^{s-1}]}+p^kG\subseteq\beta+p^kG
\end{align*}
\noindent Of particular interest will be those subgroups $K$ where $[K+p^iG:p^iG]$ is infinite for some $i$. As we will see later, if this is not the case then congruence conditions involving $K$ can be reduced to a finite union of congruence conditions not involving $K$, which we can then separate into individual congruence conditions at the cost of moving to $acl^{eq}(\emptyset)$. In particular, we would like to understand the following:
\begin{itemize}
	\item How does $[K+p^rG:p^rG]$ relate to $([K+p^iG:p^iG]: {i<r})$
	\item How does $[(K+p^rG)\cap(p^{r-1}G):p^rG]$ relate to $[K+p^rG:p^rG]$ and $[p^{r-1}G:p^rG]$
	\item How does $[p^{r-1}G:p^rG]$ relate to $[G:p^rG]$ when it is non-trivial
\end{itemize}
First, as a reminder, we have the following:
\begin{lemma*}
	Let $D\leq G$ be a convex subgroup, and $s\geq r$. Then $$D^{[p^s]}\cap p^rG = p^rD^{[p^{s-r}]}$$
\end{lemma*}
\noindent We now continue to prove various other supplementary facts.

\begin{proposition}\label{infiniterightquotientsconvex}
	Let $r\geq 1$. Then $\frac{G}{pG}\cong\frac{p^{r-1}G}{p^rG}$, and hence $[G:p^rG]=[G:pG]^r$
\end{proposition}
\begin{proof}
	The map $x+pG\mapsto p^{r-1}x+p^rG$ is an isomorphism since $G$ is torsion-free. The second part follows from the fact that $$[G:p^rG]=\prod\limits_{1\leq i \leq r}[p^{i-1}G:p^iG]$$
\end{proof}
\begin{note}
	From this we also observe that $[G:pG] = \infty \iff [G:p^rG] = \infty$.
\end{note}

\noindent In order to deal with the questions involving $\frac{K+p^rG}{p^rG}$, it will be helpful to consider the following arrangement of subgroups:
\[\begin{tikzcd}
	{K+p^rG} & {pG} \\
	{(K\cap pG) +p^rG} & {p^2G} \\
	{(K\cap p^2G) +p^rG} & {p^{r-1}G} \\
	{(K\cap p^{r-1}G)+p^rG} & {p^rG} \\
	{p^rG}
	\arrow[no head, from=1-1, to=2-1]
	\arrow[no head, from=1-2, to=2-1]
	\arrow[no head, from=1-2, to=2-2]
	\arrow[no head, from=2-1, to=3-1]
	\arrow[no head, from=2-2, to=3-1]
	\arrow[dotted, no head, from=2-2, to=3-2]
	\arrow[dotted, no head, from=3-1, to=4-1]
	\arrow[no head, from=3-2, to=4-1]
	\arrow[no head, from=3-2, to=4-2]
	\arrow[no head, from=4-1, to=5-1]
	\arrow[no head, from=4-2, to=5-1]
\end{tikzcd}\]
Noting that $(K+p^rG)\cap p^iG =  (K\cap p^iG)+p^rG$ for $i\leq r$. We will prove the following:
\begin{proposition}\label{DescendingInfiniteQuotients}
	If $K\leq G$ is either convex or of the form $\alpha^{[p^s]}$, then
	$$[K+p^rG:p^rG]=\infty \iff [(K\cap p^{r-1}G)+p^rG:p^rG] = \infty$$
\end{proposition}
\begin{proof}
    We split the proof into two cases; first the case where $K$ is convex in \Cref{leftconvexrelationship}, and then the case where $K=\alpha^{[p^s]}$ in \Cref{leftconvexnonconvexrelationship2}
\end{proof}
\begin{lemma}\label{leftconvexrelationship}
	Suppose $ K \leq G$ is convex. Then $$\frac{(K\cap p^{r-1}G) + p^rG}{p^rG}\cong\frac{K}{pK}$$ and $$[K+p^rG:p^rG]= [K:pK]^r$$ and hence $$[K+p^rG:p^rG]=\infty \iff [(K\cap p^{r-1}G)+p^rG:p^rG] = \infty$$
\end{lemma}
\begin{proof}
	First, note that
	\begin{align*}
		\frac{(K\cap p^{r-1}G) +p^rG}{p^rG} \cong& \frac{K\cap p^{r-1}G}{(K\cap p^{r-1}G)\cap p^rG} \\
		=&\frac{p^{r-1}K}{K\cap(p^{r-1}G \cap p^rG)} &\textrm{as $K$ is convex, hence pure}\\
		=&\frac{p^{r-1}K}{p^rK} \\
		\cong&\frac{K}{pK} &\textrm{as $G$ is torsion-free}
	\end{align*}
	And also that
	\begin{align*}
		\frac{K+p^rG}{p^rG} \cong& \frac{K}{p^rK}
	\end{align*}
	Then by \Cref{infiniterightquotientsconvex} we get that $$[K+p^rG:p^rG]=[(K\cap p^{r-1}G)+p^rG:p^rG]^r$$
\end{proof}
\noindent Next, let us deal with the case where $K=\alpha^{[p^s]}$. First, we prove a general fact about quotients that appear in the tower of subgroups:
\begin{lemma}
	Let $\alpha\leq G$ be convex, $0\leq i < j\leq r\leq s$. Then $$\frac{(\alpha^{[p^s]}\cap p^iG) + p^rG}{(\alpha^{[p^s]}\cap p^jG)+p^rG} \cong \frac{\alpha^{[p^{s-i}]}+p^{j-i}G}{p^{j-i}G}$$
\end{lemma}
\begin{proof}
	\begin{align*}
		\frac{(\alpha^{[p^s]}\cap p^iG) + p^rG}{(\alpha^{[p^s]}\cap p^jG) + p^rG} \cong& \frac{\alpha^{[p^s]}\cap p^iG}{\alpha^{[p^s]}\cap p^iG\cap(\alpha^{[p^s]}\cap p^jG + p^rG)} \\
		=& \frac{p^i\alpha^{[p^{s-i}]}}{\alpha^{[p^s]}\cap p^jG} & \tx{by \Cref{keylemma}} \\
		=& \frac{p^i\alpha^{[p^{s-i}]}}{p^i(\alpha^{[p^{s-i}]}\cap p^{j-i}G)} \cong \frac{\alpha^{[p^{s-i}]}}{\alpha^{[p^{s-i}]}\cap p^{j-i}G} \\
		\cong& \frac{\alpha^{[p^{s-i}]} +p^{j-i}G}{p^{j-i}G}
	\end{align*}
\end{proof}
\noindent From this, we get the following equation for the index of $p^rG$ in $K+p^rG$:
\begin{corollary}\label{leftconvexnonconvexrelationship1}
    If $\alpha\leq G$ is convex, with $s\geq r$, then
	\begin{align*}
	    [\alpha^{[p^s]}+p^rG:p^rG] = \prod\limits_{0\leq i< r}[\alpha^{[p^{s-i}]}+pG:pG]
	\end{align*}
\end{corollary}
\begin{proof}
	\begin{align*}
		[\alpha^{[p^s]}+p^rG:p^rG] =& \prod\limits_{0\leq i< r}[(\alpha^{[p^s]}\cap p^iG)+pG:(\alpha^{[p^s]}\cap p^{i+1}G)+pG] \\
		=& \prod\limits_{0\leq i< r}[\alpha^{[p^{s-i}]}+pG:pG]
	\end{align*}
\end{proof}
\noindent Using this equation, we can bound the full index by a power of the final index, and hence complete the proof of \Cref{DescendingInfiniteQuotients}:
\begin{corollary}\label{leftconvexnonconvexrelationship2}
	$$[\alpha^{[p^s]}+p^rG:p^rG] \leq [(\alpha^{[p^s]}\cap p^{r-1}G)+p^rG:p^rG]^r$$
	And hence
	\begin{align*}
		&[\alpha^{[p^s]}+p^rG:p^rG] =\infty \\
		\iff& [(\alpha^{p^s}\cap p^{r-1}G)+p^rG:p^rG] = \infty
	\end{align*}
\end{corollary}
\begin{proof}
	Note that if $0<i<r$, then $\alpha^{[p^{s-i}]}\subseteq \alpha^{[p^{s-r+1}]}$, and so we have
	\begin{align*}
		[(\alpha^{p^s}\cap p^{r-1}G)+p^rG:p^rG] =& [\alpha^{[p^{s-r+1}]}+pG:pG] \\
		\geq& [\alpha^{[p^{s-i}]}+pG:pG] \\
	\end{align*}
	The result then follows immediately from \Cref{leftconvexnonconvexrelationship1}.
\end{proof}

\noindent Now, we have enough to prove \Cref{main} in the remaining case:
\begin{theorem}\label{prG uniformisable version}
    Let $d,e\in\mathbb{N}_{>0}$, $D\leq G$ a definable convex subgroup, $B\subseteq G$, and $f:G^d\rightarrow (\frac{G}{D+p^rG})^{(e)}$ be a $B$-definable function. Then there are:
    \begin{itemize}
        \item $k,N,n\in \mathbb{N}$, $B$-definable sets $\{X_i\subseteq G^d\}_{i\leq k}$, $n_i,m_i\in \mathbb{N}$ for $i\leq k$
        \item  $B$-definable linear functions $f_{i,v},g_{i,w}:G^d\rightarrow G$, for $i\leq k$, $v\leq n_i$, $w\leq m_i$
        \item Definable subgroups $H_{i,v},K_{i,w}$ which are either of the form $\alpha+p^sG$ or $\alpha^{[p^t]}+p^sG$ for some $s\leq r, t$ and $\alpha\in \Spn$, with $[\alpha+p^rG:D+p^rG], [\alpha^{[p^t]}+p^rG:D+p^rG]< N$, and $[H_{i,0}:D+p^rG]<N$. 
    \end{itemize}
    Such that $f = \bigcup\limits_{i\leq k}((\bigcap\limits_{v\leq n_i}f_{i,v}^{H_{i,v}}|_{X_i})\setminus(\bigcup\limits_{w\leq m_i}g_{i,w}^{K_{i,w}}|_{X_i}))$.
\end{theorem}
\begin{proof}
    By \cite{QE}, we may assume that $f$ is defined by a formula of the form
	\begin{align*}
		\bigvee\limits_{i}(\exists\bar\theta)(\xi_i(\bar\theta, \bar\eta)\wedge \psi_i(\bar x, y, \bar\theta))
	\end{align*}
	Where $\xi_i$ is a formula restricted to an $n$-spine, and $\psi_i$ is a conjunction of literals, with each atom of the form $r_iy\square_\alpha t_i(\bar x)$ where $t_i(\bar x)$ is a linear term and $\square \in \{<, \leq, >, \geq, \equiv_n, \equiv_n^{[m]}\}$. This $n$ gives us the $n$ in the proposition statement.
	\nl\nl
	Note that by restricting to $X_i$ defined by $(\exists y)(\exists\bar\theta)(\xi_i(\bar\theta, \bar\eta)\wedge \psi_i(\bar x, y, \bar\theta))$ we may further reduce to
	\begin{align*}
		(\exists z)(y\equiv_{p^r}z\wedge(\exists\bar\theta)(\xi(\bar\theta, \bar\eta)\wedge \psi(\bar x, z, \bar\theta)))
	\end{align*}
	And furthermore, by following the quantifier-elimination proof in \cite{QE} and making use of the Chinese Remainder Theorem, we may reduce this to 
	\begin{align*}
		(\exists\bar\theta)(\xi(\bar\theta, \bar\eta)\wedge \chi(\bar x, y, \bar\theta))
	\end{align*}
	Where $\chi$ is a conjunction of literals with atoms of the form $ry\diamond_\alpha t(\bar x)$, with $\diamond \in \{\equiv_{p^j}, \equiv_{p^j}^{[p^s]}\}$ and $1\leq j\leq r,s$. As before, we may also assume that $\models (\forall \bar x)(\forall\bar\theta)((x\in\tx{dom}(f)\wedge\xi(\bar\theta,\eta))\rightarrow(\exists^{\geq 1}y)(\chi(\bar x, y,\bar\theta)$.
    \nl\nl
    For notational simplicity we will let $H_i$ denote groups of the form $\alpha$ or $\alpha^{[p^s]}$, and if $H_i = \alpha^{[p^s]}$, then $\equiv_{H_i,p^k}$ will mean $\equiv_{p^k,\alpha}^{[p^s]}$. If $\beta\leq G$ is convex, then $\beta^{H_i} = \beta$ if $H_i=\alpha$, and $\beta^{H_i} = \beta^{[p^s]}$ if $H_i = \alpha^{[p^s]}$. $\mathcal{S}_{H_i}$ will mean $\mathcal{S}_l$ for some $l$ with $\alpha\in \mathcal{S}_l$.
	\nl\nl
    Now, as $D$ need not be an element of any of the spines, we add new quantifier-free symbols $\equiv_D$ and $\equiv_{D,p^r}$ and do our analysis in this extended language. Much as in \Cref{alpha^{[p^s]}+p^rG section}, we want to remove reference to subgroups of the form $H+p^iG\subseteq D+p^iG$. We may do this in the exact same method as in the proof of \Cref{main for a^p^s+p^rG}. Note that this both ensures that $\alpha\geq D$ and introduces formulas using the new symbols, so we also have literals $ry\diamond_Dt(\bar x)$.
    \nl\nl
	Now, we are currently in a situation where, for fixed $\bar \theta\in\xi(\mathcal{A},\bar\eta)$, we have a function $f_{\bar\theta}$ where $f_{\bar\theta}(x)$ is given by congruence conditions coming from various subgroups that form a lattice as follows:

	
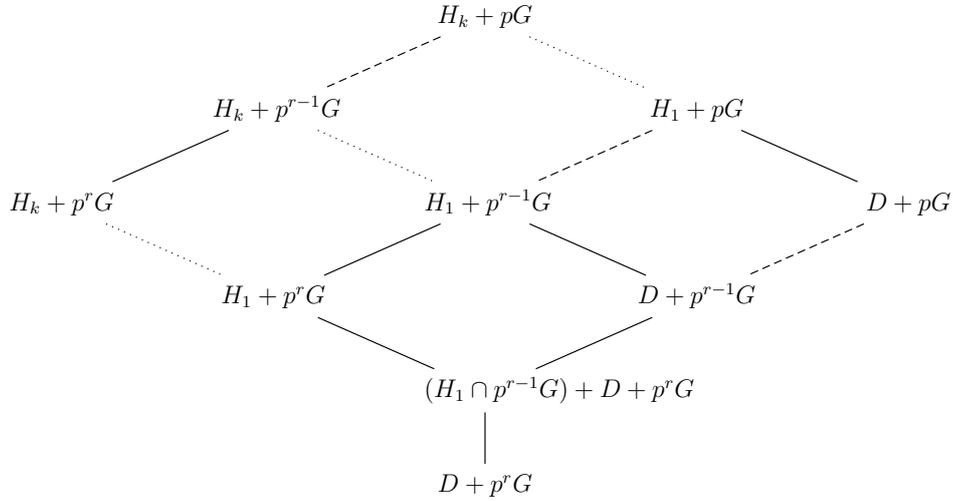
\[\begin{tikzcd}[
			scale cd=.8,
            font = \large,
			text height = height("$H_1+p^{r-1}G$"),
			text width = width("$H_1+p^{r-1}G$"),
			align = center
		]
		&& {H_k+pG} \\
		& {H_k+p^{r-1}G} && {H_1+pG} \\
		{H_k+p^rG} && {H_1+p^{r-1}G} && D+pG \\
		& {H_1+p^rG} && {D+p^{r-1}G} \\
		&& {(H_1\cap p^{r-1}G) +D+ p^rG} \\
		&& {D+p^rG}
		\arrow[dashed, no head, from=1-3, to=2-2]
		\arrow[dotted, no head, from=1-3, to=2-4]
		\arrow[no head, from=2-2, to=3-1]
		\arrow[dotted, no head, from=2-2, to=3-3]
		\arrow[dashed, no head, from=2-4, to=3-3]
		\arrow[no head, from=2-4, to=3-5]
		\arrow[dotted, no head, from=3-1, to=4-2]
		\arrow[no head, from=3-3, to=4-2]
		\arrow[no head, from=3-3, to=4-4]
		\arrow[dashed, no head, from=3-5, to=4-4]
		\arrow[no head, from=4-2, to=5-3]
		\arrow[no head, from=4-4, to=5-3]
		\arrow[no head, from=5-3, to=6-3]
	\end{tikzcd}\] and the intersection of all of these congruence conditions results in at most 
    $e$ cosets of $D+p^rG$ - note that some of the refinements earlier may have 
    reduced the number of cosets, so we cannot say that it is equal to $e$. 
    First, note that if $[G:D+p^rG]<\infty$, the result is immediate. So we may 
    assume that $[G:D+p^rG]=\infty$, and hence $[D+p^{r-1}G:D+p^rG]=\infty$ also.
    We will show that there is some $i$ with $[H_i+p^rG:D+p^rG]$ finite such that 
    a positive congruence relation of the form $ry\equiv_{H_i,p^r}t(\bar x)$ 
    appears in the definition, with $p\nmid r$ and $t$ a $\mathbb{Z}$-linear term. As in the proof of \Cref{main for a^p^s+p^rG} we first ensure that each congruence condition either restricts to or excludes exactly 1 coset of the appropriate subgroup, at the cost of the terms being $\mathbb{Z}[\frac{1}{p}]$-linear for congruences coming from subgroups of the form $H+p^iG$ with $i < r$.
    \nl\nl
    First, suppose that $[H_1 + p^rG:D+p^rG] = \infty$. Then, by \Cref{leftconvexnonconvexrelationship2} and \Cref{leftconvexrelationship} $[(H_1\cap p^{r-1}G) +D+ p^rG:D+p^rG] =\infty$ also. But then if there is no positive $D+p^rG$ congruence then for this $\bar\theta$, this intersection of congruence conditions results in an infinite number of $D+p^rG$ cosets intersected with a cofinite number of $D+p^rG$ cosets, which is hence infinite. So, there must be a positive $D+p^rG$ congruence.
    \nl\nl
    Let $j$ be maximal such that $[H_j+p^rG:D+p^rG]$ is finite. Consider $f_{\bar\theta, j}(x) = f_{\bar\theta}(x) +H_j$. We obtain (with some uniform case distinctions) a definition for $f_{\bar\theta, j}$ coming from the definition for $f_{\bar\theta}$ by once again making use of \Cref{QE 3.2,QE 3.3} as in the proof of \Cref{main for a^p^s+p^rG}; specifically, this definition contains no new terms, makes no reference to any new subgroups, and makes no reference to $H_l$ for $l<j$.
    \nl\nl
    If $H_j$ is a convex subgroup of $G$, then by the above this definition must contain a positive congruence condition of the form $ry\equiv_{H_j,p^r}t(\bar x)$ which is a lift of a congruence condition $ry\equiv_{H_i,p^r}t(\bar x)$ for some $i \leq j$, and so $f_{\bar\theta}$ satisfies the proposition, and any congruence condition coming from $H_l$ with $l>i$ is redundant beyond potentially imposing a consistency condition, which can also be ignored by the earlier restriction that $\models(\exists^{\geq 1}y)(\chi(\bar x, y, \bar\theta))$.
    \nl\nl
    On the other hand, if $H_j$ is of the form $\alpha^{[p^s]}$ for some $\alpha\geq D$, then we have that $[\alpha+p^rG:D+p^rG]<\infty$, and also by \Cref{main for a^p^s+p^rG} we obtain a positive congruence of the form $ry\equiv_{\alpha, p^r}^{[p^t]}$ for some $t\leq s$, where $\alpha^{[p^t]} = H_i$ for some $i$. Now, we may assume that $i$ is minimal such that there is a positive congruence of the form $H_i$.
    \nl\nl
    If $[H_i+p^rG:D+p^rG] =\infty$, then $[(D+p^{r-1}G)\cap (H_i+p^rG):D+p^rG]=\infty$ by \Cref{leftconvexnonconvexrelationship2}. Now, the positive congruence conditions we obtain from \Cref{main for a^p^s+p^rG} come from subgroups $H_i+p^rG,K_1+p^{l_1}G,\dots,K_m+p^{l_m}G$ with $l_i < r$. Hence, $D+p^{r-1}G\subseteq \bigcap\limits_{j\leq m}(K_j+p^{l_j}G)$. But then we have that $[(H_i+p^rG)\cap\bigcap\limits_{j\leq m}(K_j+p^{l_j}G):D+p^rG]=\infty$, and so $[(H_i+p^rG)\cap\bigcap\limits_{j\leq m}(K_j+p^{l_j}G):H_j+p^rG]=\infty$ as well, but this contradicts \Cref{main for a^p^s+p^rG}. So we must have that $[H_i+p^rG:D+p^rG] <\infty$. Hence $f_{\bar\theta}$ satisfies the proposition, and the congruence conditions coming from $H_l$ with $H_l = \beta^{H_l}$ and $\beta > \alpha$ are again redundant beyond imposing a consistency condition, which can again be ignored.
	\nl\nl
    Now, we have that each $f_{\bar\theta}$ satisfies the theorem statement, however $f$ is a family-union of the various $f_{\bar\theta}$; in order to extend this to $f$ we start by noting that there is a uniform bound on $N$; This is a straightforward application of compactness: for each $\bar\theta$ in the family there is some finite bound, so by compactness there must be a uniform bound for the entire family.
    Now, note that knowing $[H_i+p^rG:D+p^rG] = N<\infty$ fully determines $H_i+p^rG$, and hence also $H_i+p^jG$ for $j\leq r$, since we have that if $\alpha <\beta$ are convex subgroups then
    \begin{align*}
        \alpha+p^rG\leq \dots\leq \alpha^{[p^2]}+p^rG\leq\alpha^{[p]}+p^rG\leq\beta+p^rG
    \end{align*}
    so there is at most one such extension of $D+p^rG$ of given finite index.
    Hence, we see that there are only finitely many possible configurations for $f_{\bar\theta}$, each of which is definable by the above note, and so by refining the domain we see that $f$ satisfies the theorem statement also.    
\end{proof}

\begin{corollary}\label{nonuniform main for prG}
	Let $d,e\in\mathbb{N}_{>0}$, and $f:G^d\rightarrow (\frac{G}{p^rG})^{(e)}$ be a $B$-definable function. Then we can find some $j\in \mathbb{N}$, $\acl^{eq}(B)$-definable sets $\{Y_i\}_{i\leq j}$, $\acl^{eq}(B)$-definable linear functions $f_i:X_i\rightarrow \frac{G}{p^rG}$ such that $f = \bigcup_if_i|_{Y_i}$.
\end{corollary}
\begin{proof}
	Let $X_i, H_{i,v}, K_{i,v}, f_{i,v}, g_{i,v}$ be as in \Cref{prG uniformisable version}, and let $f_i$ be the $i^{th}$ summand of the obtained definition for $f$. As remarked at the end of the proof of \Cref{prG uniformisable version}, if $[H_{i,0}+p^rG:p^rG]=[H_{i,j}+p^rG:p^rG]=M_i$, then $H_{i,0}+p^rG=H_{i,j}+p^rG = \{d_{i,l}+p^rG\}_{l<M_i}$, and $d_{i,l}+p^rG \in\acl^{eq}($\cs$D$\ce$)$, as witnessed by the formula
	\begin{align*}
		(\exists \beta\in \mathcal{S}_H)(&[\beta^H+p^rG:D+p^rG] = M_i \\
		&\wedge d_{i,l}+D+p^rG \subseteq \beta^H+p^rG
	\end{align*}
	And $v\equiv_{H_i,p^r}w$ is equivalent to $\bigvee\limits_{l<M_i}v\equiv_{D,p^r}d_{i,l}+w$, and so there are a finite number of $\acl^{eq}($\cs$D$\ce$)\cup B$-definable linear congruence relations of the form $m_iy \equiv_{D,p^r}t_i(\bar x)+g_{i,l}$, with $p\nmid m_l$, such that on $X_i$ we have $y\in f_i(\bar x)\rightarrow \bigvee m_iy \equiv_{D,p^r}t_i(\bar x)+g_{i,l} $, and so we can refine to the definable sets $X_{i,l} := \{\bar x :  (\forall y) ([m_iy \equiv_{D,p^r}t_i(\bar x)+g_{i,l}]\rightarrow y\in f_i(\bar x)\}$; each of these definable sets is $\acl^{eq}($\cs$D$\ce$)\cup B$-definable, and clearly $f =\bigcup f_{i,l}|_{X_{i,l}}$ where $f_{i,l}(x) = y\iff m_iy \equiv_{p^r}t_i(\bar x)+g_{i,l}$.
	
\end{proof}

\begin{corollary}\label{uniform main for prG}
    Let $d,e\in\mathbb{N_{>0}}$, $\mathcal{D}$ a definable family of convex subgroups, and $\left\{f_D:G^d\rightarrow(\frac{G}{D+p^rG})^{(e)}\right\}_{D\in\mathcal{D}}$ a definable family of functions. Then there is a partition of $\mathcal{D}$ into definable sets $\mathcal{D}_1,\dots,\mathcal{D}_\ell$ such that \Cref{prG uniformisable version} holds uniformly in $D$ on each $\mathcal{D}_i$ (with each $X_i$ being replaced by a $B$-definable family of sets $X_i^D$).
\end{corollary}
\begin{proof}
    First, note that if $\alpha\in\Sn$ with $[\alpha+p^jG:D+p^jG] = k$ for some $j,k\in\mathbb{N}$, then $[\alpha+p^rG:D+p^rG] = k^{\frac{r}{j}}=M$, and so since convex subgroups are totally ordered by inclusion $\alpha+p^rG$ is definable from $D+p^rG$ by
    \begin{align*}
        (\exists\alpha\in\Sn)(\exists x_1,\dots,x_M)(&\bigwedge\limits_{1\leq i\leq M} x_i\in\alpha+p^rG \\
        &\wedge\bigwedge\limits_{i\neq j}x_i-x_j\notin D+p^rG \\
        &\wedge (\forall z)(z\in \alpha+p^rG \rightarrow \bigvee\limits_{1\leq i\leq M}x_i-z\in D+p^rG)
    \end{align*}
    and hence so is $\alpha+p^iG$ for any $i\leq r$. If $[\alpha^{[p^s]}+p^jG:D+p^jG] = k$, then we get that $\alpha+p^iG$ is definable from $D+p^rG$ for any $i\leq r$, and then $\alpha^{[p^s]}+p^iG$ is definable from $\alpha+p^iG$.
    \nl\nl
    So we see that the conditions in \Cref{prG uniformisable version} are definable from $D$, and so by compactness there is a uniform bound for $N$ - if not, then there is some $D\in\mathcal{D}$ with $f_D$ not of the form given by \Cref{prG uniformisable version}. So, as each of the functions we are looking for comes from the terms in the definition, we only have finitely many possible configurations which give a function $g_D$ compatible with $f_D$, and so we may partition $\mathcal{D}$ into sets defined by which combination of these configurations and case distinctions gives rise to $f_D$. Note that the case distinctions are what give rise to the uniform families $X_i^D$.
\end{proof}




\section{Examples}\label{examples section}
In this section, we go over some examples to show why the given theorems cannot be improved the the best imaginable version (that is, being uniformly piecewise linear) in all of the cases. To start with, we give a method to generate ordered abelian groups with certain useful properties.
\subsection{Patterns in $\dim _{\mathbb{F}_p}\frac{\alpha^{[p^s]}+pG}{\alpha^{[p^{s+1}]}+pG}$}
Let $\mathbb{P}$ denote the set of primes. We show that for any function $f:\mathbb{P}\times \mathbb{N}_{\geq 1}\rightarrow\mathbb{N}\cup\{\infty\}$, there is a group $G$ such that $\dim _{\mathbb{F}_p}\frac{\alpha^{[p^s]}+pG}{\alpha^{[p^{s+1}]}+pG} = f(p,s)$

\begin{lemma}\label{Base example for dimensions}
	For $p\in\mathbb{P}$, $n\in \mathbb{N}_{\geq 1}$, let $G_{p^n} = \{q_0+q_1t + \dots q_mt^m=q(t)\in \mathbb{Z}[t] : p^n |\sum q_i\}$. Then for $\alpha = \{0\}$
	
	\[ \dim _{\mathbb{F}_p}\frac{\alpha^{[p^s]}+pG_{p^n}}{\alpha^{[p^{s+1}]}+pG_{p^n}} = \begin{cases} 
		0 & s\geq 1, s\neq n \\
		1 & s = n 
	\end{cases}
	\]
\end{lemma}
\begin{proof}
    First, note that $\alpha^{[p^s]} = \{\sum q_i t^i : p^s | q_i \wedge p^n|\sum q_i\}$. So, if $s > n$ this reduces to $\{\sum q_i t^i : p^s | q_i\}\subseteq \{\sum q_it^i : p^{n+1}|q_i\}\subseteq \{\sum q_it^i : p | q_i \wedge p^{n+1} |\sum q_i\} =pG_{p^n}$. This gives us the case for $s > n$.
    
    In the case where $s < n$, let $q(t)=q_0+\dots+q_mt^m\in \alpha^{[p^s]}$.
    Let $u_i = (p-1)q_i$ for $0\leq i \leq m$, and $u_{m+1} = -\sum\limits_{0\leq i \leq m} u_i$. Let $u(t) = \sum\limits_{0\leq i\leq
    m+1}u_it^i$; note that $u(t)\in pG_{p^n} $. Also note that $p^n | u_{m+1}$, and hence $p^{s+1}|u_{m+1}$, and so we have that $u(t) + q(t) \in \alpha^{[p^{s+1}]}=\{v(t) : p^{s+1}|v_i \wedge p^n|\sum v_i\}$.
    \nl\nl
    Finally, we deal with the case where $s=n$. Let $f:
    \alpha^{[p^s]}+pG_{p^n}$ be defined by $f(q(t)) = \frac{\sum q_i}{p^n}
    \mod p$. Note that if $f(q(t)) = 0$, then $p^{n+1}|\sum q_i$ and $p| q_i$,
    so $q_i \in pG_{p^n}$. Also, note that $f(rp^n)=r$. Hence, this descends
    to a group isomorphism $\tilde{f}:\frac{\alpha^{[p^n]}+pG_{p^n}}
    {\alpha^{[p^{n+1}]}+pG_{p^n}}\rightarrow \frac{\mathbb{Z}}{p\mathbb{Z}}$
\end{proof}
\begin{lemma}\label{General example for dimensions}
	Let $\mathbb{N}_{\geq 1} = \coprod\limits_{i\in\omega} U_i$ with each
    $U_i$ infinite, and let $s_i(q(t)) = \sum\limits_{j\in U_i} q_j$. Let $G = 
    \{q(t)\in \mathbb{Z}[t] : (\forall i)( p_i^{n_i} | s_i(t))\}$ where
    $p_i\in\mathbb{P}, n_i\in \mathbb{N}_{\geq 1}$, then for $\alpha = \{0\}$ 
    we have that:
	\begin{align*}
	    \dim _{\mathbb{F}_p}\frac{\alpha^{[p^s]}+pG}{\alpha^{[p^{s+1}]}+pG} = |\{i : p_i = p\wedge n_i = s\}|
	\end{align*}
\end{lemma}
\begin{proof}
    This is a straightforward modification of the proof of \Cref{Base example for dimensions}
\end{proof}
\subsection{An example of a definable function that is not piecewise linear}
Let $\mathbb{N}_{\geq 1} = \coprod\limits_{i\in\omega} U_i$ with each $U_i$
infinite, then set $G = \{q(t)\in \mathbb{Z}[t] : (\forall i)( p^2 | 
s_i(t))\}$, where $s_i$ is defined as above. We show that there is a definable
partial function $f:G\rightarrow \frac{G}{\{0\}^{[p^3]}+p^2G}$ for which there
does not exist a finite partition of definable sets $\{X_i\}_{i< n}$ and
definable linear functions $f_i:X_i\rightarrow \frac{G}{\{0\}^{[p^3]}+p^2G}$
such that $f = \bigcup\limits_{i < n} f_i|_{X_i}$.
\begin{claimproof}
	Let $\alpha = \{0\}$. Then note that by \Cref{General example for dimensions}
	\[ \dim_{\mathbb{F}_p}\frac{\alpha^{[p^s]}+pG}{\alpha^{[p^{s+1}]}+pG} = 
    \begin{cases}
		0 & s\neq 2 \\
		\infty & s = 2 
	\end{cases}
	\]
	Now, by \Cref{Describing a^p^j - a^p^s quotients}, \[
	\frac{(\alpha^{[p^3]}+pG)\cap(\alpha^{[p^2]}+p^2G)}{\alpha^{[p^3]}+p^2G} = 
    \frac{\alpha^{[p^{3-1}]}\cap pG+\alpha^{[p^3]}+p^2G}{\alpha^{[p^3]}+p^2G} 
    \cong \frac{\alpha^{[p^1]}+pG}{\alpha^{[p^2]}+pG}
	\]
	And so $(\alpha^{[p^3]}+pG)\cap(\alpha^{[p^2]}+p^2G) = 
    \alpha^{[p^3]}+p^2G$. We also have that \[\dim_{\mathbb{F}
    _p}\frac{\alpha^{[p^2]}+p^2G}{\alpha^{[p^3]}+p^2G} = \dim_{\mathbb{F}
    _p}\frac{\alpha^{[p^2]}+pG}{\alpha^{[p^3]}+pG} = \infty\]
	Furthermore, by \Cref{infiniterightquotientsnonconvex} we have that \[ 
    \dim_{\mathbb{F}_p} \frac{\alpha^{[p^2]}+pG}{\alpha^{[p^2]}+p^2G} = 
    \dim_{\mathbb{F}_p}\frac{\alpha^{[p^3]}+pG}{\alpha^{[p^3]}+p^2G} = \infty\]
	Now, we are in a position to define the function. First, let
	\begin{align*}
		f_1:G&\rightarrow\frac{G}{\alpha^{[p^2]}+p^2G} & f_2:G&\rightarrow\frac{G}{\alpha^{[p^3]}+pG}\\
		x&\mapsto x+\alpha^{[p^2]}+p^2G &x&\mapsto \alpha^{[p^3]}+pG
	\end{align*}
	And then let $f= f_1\cap f_2$ be the induced partial function
    $f:G\rightarrow \frac{G}{\alpha^{[p^3]}+p^2G}$. We show that $f$ is not
    piecewise linear. First, note that the domain of $f$ is \begin{align*}
        (\alpha^{[p^3]}+pG) + (\alpha^{[p^2]}+p^2G) = \alpha^{[p^2]}+pG
    \end{align*}
    Now, suppose that we have some linear piece $h_i(x) = \frac{1}{b_i}(a_ix + 
    g_i) + \alpha^{[p^3]}+p^2G$ with $g_i\in G,a_i,b_i\in \mathbb{Z}$ with
    $a_i,b_i$ coprime and $b_i\neq 0$.
	\nl\nl
	First, suppose that $p|a_i$. Then $p\nmid b_i$, and so there is some $q\in 
    \mathbb{Z}$ with $qb_i \equiv 1 \mod p^2$, and then $h_i(x) = qa_i(x) + 
    qg_i + \alpha^{[p^3]}+p^2G$, so we may assume that $b_i = 1$. Now, note
    that for $x \in \alpha^{[p^2]}+pG$,
    \begin{align*}
        px \in p\alpha^{[p^2]}+p^2G = 
        \alpha^{[p^3]}\cap pG + p^2G \subseteq \alpha^{[p^3]}+p^2G
    \end{align*}
	Hence, $h_i(x) = g_i + \alpha^{[p^3]}+p^2G$. But then, as $h_i(x) 
    \subseteq f_1(x)$, we must have that $g_i \equiv_{\alpha, p^2}^{[p^2]} x$.
    Hence, the domain of $f_i$ is contained in $g_i + \alpha^{[p^2]}+p^2G$.
    \nl\nl
	Next, suppose that $p\nmid a$. Then, there is some $r$ such that $ar 
    \equiv 1 \mod p^2$, and hence
    \begin{align*}
        brh_i(x) = x + rg + \alpha^{[p^3]}+p^2G 
    \end{align*}
    \noindent However, since $h_i(x) \subseteq f_2(x)$, we have that
    $h_i(x),brh_i(x)\subseteq \alpha^{[p^3]}+pG$. And so we must have that $x 
    \equiv_{\alpha,p}^{[p^3]} -rg $, so the domain of $h_i$ is contained in a
    single coset of $\alpha^{[p^3]}+pG$.
	\nl\nl
	Hence, we see that any linear piece has domain contained in a single coset
    of either $\alpha^{[p^3]}+pG$ or $\alpha^{[p^2]}+p^2G$. But then any union
    of a finite number of linear pieces has domain strictly smaller than
    $(\alpha^{[p^3]}+pG)+(\alpha^{[p^2]}+p^2G)$ since both of these groups are
    infinite index in the sum.
\end{claimproof}

\subsection{Uniformity for families of subgroups of the form $D+p^rG$}
The difference between the non-uniform \Cref{nonuniform main for prG} and the
uniform \Cref{uniform main for prG} is rather stark, with \Cref{nonuniform main for prG} being of a much simpler form, but introducing a source of
nonuniformity coming from $\acl^{eq}$; one might wonder if it is possible to
get a uniform version of \Cref{nonuniform main for prG}, but the following
example shows that this is not possible:
\nl\nl
Let $G=\bigoplus_\omega\mathbb{Z}_{(2)}$, ordered lexicographically. Then $\mathcal{S}_2 \cong(\omega+1)^*$, where $L^*$ denotes the reverse order on
$L$. Let us write
\begin{align*}
    \mathcal{S}_2 = \{\{0\}, \dots, \alpha_{n+1},\alpha_n,\dots,\alpha_1,\alpha_0\}
\end{align*}
Then we have that for $i\geq 1$, $\frac{\alpha_{i-1}+2G}{\alpha_i+2G}
\cong\frac{\mathbb{Z}}{2\mathbb{Z}}$, and we have a definable family of
functions given by
\begin{align*}
    f_{\alpha_i}(x) = y+\alpha_i+2G \iff (y-x \in \alpha_{i-1}+2G)\wedge(y-x \notin \alpha_{i}+2G)
\end{align*}
To show that this cannot be described as a uniform projection of a linear
function, even after partitioning $\mathcal{S}_2\setminus\{\alpha_0,\{0\}\}$
into finitely many sets, first note that one of the partition sets would need
to be co-initial in $\mathcal{S}_2$. Next, note that $\frac{G}{\alpha_i+2G}$
can be seen as binary sequences of length $i$, and $\frac{G}{2G}$ can be seen
as finite length binary sequences; under this interpretation, $f_{\alpha_i}$
flips the first digit. It is clear that the linear function must hence be of
the form $f_{\alpha_i}(x) = x+g+\alpha_i+2G$.
To ensure that $f_{\alpha_i}$ is as required, we must have $g$ with
$g+\alpha_i+2G = (1,0,\dots,0)$. However, in order for a single $g$ to work for
$\alpha_i$ and $\alpha_j$ with $i< j$, then $g+2G$ must have a $1$ in both the
$i^{th}$ and the $j^{th}$ position of the binary sequence, but the $1$ in the
$i^{th}$ position contradicts the requirement for $f_{\alpha_j}$.


\printbibliography[title={References}]



\end{document}